\documentclass{alggeom}
\usepackage{mathtools, microtype, mathrsfs, xfrac, amssymb, enumerate, xspace, stmaryrd, amsthm}
\usepackage[alphabetic]{amsrefs}
\usepackage{ascii}
\usepackage[latin1]{inputenc}
\usepackage{tikz-cd}
\usepackage[all,cmtip]{xy}
\usepackage{hyperref}
\hypersetup{
	colorlinks=true,
	linkcolor=blue,
	filecolor=magenta,      
	urlcolor=cyan,
}
\usepackage{faktor}
\usepackage{cleveref}
\usepackage{comment}
\usepackage{graphicx}
\usepackage{aliascnt}
\usepackage{float}

\numberwithin{equation}{section}

\numberwithin{subsection}{section}

\allowdisplaybreaks[1]

\newtheorem{theorem}{Theorem}[section]
\newtheorem{proposition}[theorem]{Proposition}
\newtheorem{proposition-definition}[theorem]
{Proposition-Definition}
\newtheorem{corollary}[theorem]{Corollary}
\newtheorem{lemma}[theorem]{Lemma}

\theoremstyle{definition}
\newtheorem{definition}[theorem]{Definition}

\newtheorem{example}[theorem]{Example}

\newtheorem{remark}[theorem]{Remark}

\theoremstyle{remark}

\newcommand\cA{\mathcal{A}} 
\newcommand\cC{\mathcal{C}} 
 
 \newcommand\cH{\mathcal{H}}
 
 \newcommand\cL{\mathcal{L}}
\newcommand\cM{\mathcal{M}} 
\newcommand\cO{\mathcal{O}} \newcommand\cP{\mathcal{P}}

 \newcommand\cX{\mathcal{X}}
\newcommand\cY{\mathcal{Y}} \newcommand\cZ{\mathcal{Z}}

\renewcommand\AA{\mathbb{A}} 
\newcommand\CC{\mathbb{C}}  
\newcommand\GG{\mathbb{G}} \newcommand\HH{\mathbb{H}}

 \newcommand\PP{\mathbb{P}}

 \newcommand\VV{\mathbb{V}}
 
 \newcommand\ZZ{\mathbb{Z}}

\newcommand\arr{\ifinner\to\else\longrightarrow\fi}

\newcommand\arrto{\ifinner\mapsto\else\longmapsto\fi}

\newcommand\larr{\longrightarrow}

\newcommand{\hooklongrightarrow}{\lhook\joinrel\longrightarrow}

\renewcommand\H{\operatorname{H}}

\newcommand\into{\hookrightarrow}

\newcommand\im[1]{\operatorname{im}(#1)}

\def\displaytimes_#1{\mathrel{\mathop{\times}\limits_{#1}}}

\def\displayotimes_#1{\mathrel{\mathop{\bigotimes}\limits_{#1}}}

\renewcommand\hom{\operatorname{Hom}}

\newcommand\aut{\operatorname{Aut}}

\newcommand\spec{\operatorname{Spec}}

\newcommand\codim{\operatorname{codim}}

\newcommand\id{\mathrm{id}}

\newlength{\ignora}

\renewcommand{\setminus}{\smallsetminus}

\newcommand{\mmu}{\boldsymbol{\mu}}

\newcommand{\GL}{\mathrm{GL}}

\newcommand{\PGL}{\mathrm{PGL}}

\DeclareFontFamily{U}{mathx}{\hyphenchar\font45}
\DeclareFontShape{U}{mathx}{m}{n}{
	<5> <6> <7> <8> <9> <10>
	<10.95> <12> <14.4> <17.28> <20.74> <24.88>
	mathx10
}{}
\DeclareSymbolFont{mathx}{U}{mathx}{m}{n}
\DeclareFontSubstitution{U}{mathx}{m}{n}
\DeclareMathAccent{\widecheck}{0}{mathx}{"71}
\DeclareMathAccent{\wideparen}{0}{mathx}{"75}

\renewcommand{\epsilon}{\varepsilon}

\newcommand{\Mbar}{\overline{\cM}}

\newcommand{\Mtilde}{\widetilde{\mathcal M}}

\newcommand{\Htilde}{\widetilde{\cH}}

\newcommand{\Hbar}{\overline{\cH}}

\newcommand{\mt}{\widetilde{\mathcal M}}

\begin{document}

\title{Hyperelliptic $A_r$-stable curves \\(and their moduli stack)}
\author{Michele Pernice}
\email{mpernice(at)kth.se}
\address{KTH, Room 1642, Lindstedtsv\"agen 25,
114 28, Stockholm }


\begin{abstract}
	This paper is the second in a series of four papers aiming to describe the (almost integral) Chow ring of $\Mbar_3$, the moduli stack of stable curves of genus $3$. In this paper, we introduce the moduli stack $\Htilde_g^r$ of hyperelliptic $A_r$-stable curves and generalize the theory of hyperelliptic stable curves to hyperelliptic $A_r$-stable curves. In particular, we prove that $\Htilde_g^r$ is a smooth algebraic stacks which can be described using cyclic covers of twisted curves of genus $0$ and it embeds in $\Mtilde_g^r$ (the moduli stack of $A_r$-stable curves) as the closure of the moduli stack of smooth hyperelliptic curves.
\end{abstract}

\maketitle
\section*{Introduction}

The geometry of the moduli spaces of curves has always been the subject of intensive investigations, because of its manifold implications, for instance in the study of families of curves. One of the main aspects of this investigation is the intersection theory of these spaces, which can used to solve either geometric, enumerative or arithmetic problems regarding families of curves. In his groundbreaking paper \cite{Mum}, Mumford introduced the intersection theory with rational coefficients for the moduli spaces of stable curves. After almost two decades, Edidin and Graham introduced in \cite{EdGra} the intersection theory of global quotient stacks (therefore in particular for moduli stacks of stable curves) with integer coefficients.

To date, several computations have been carried out. While the rational Chow ring of $\cM_g$, the moduli space of smooth curves, is known for $2\leq g\leq 9$ (\cite{Mum}, \cite{Fab}, \cite{Iza}, \cite{PenVak}, \cite{CanLar}), the complete description of the rational Chow ring of $\Mbar_g$, the moduli space of stable curves, has been obtained only for genus $2$ by Mumford and for genus $3$ by Faber in \cite{Fab}. As expected, the integral Chow ring is even harder to compute: the only complete description of the integral Chow ring of the moduli stack of stable curves is the case of $\Mbar_2$, obtained by Larson in \cite{Lar} and subsequently with a different strategy by Di Lorenzo and Vistoli in \cite{DiLorVis}. It is also worth mentioning the result of Di Lorenzo, Pernice and Vistoli regarding the integral Chow ring of $\Mbar_{2,1}$, see \cite{DiLorPerVis}.

The aim of this series of four papers is to describe the Chow ring with $\ZZ[1/6]$-coefficients of the moduli stack $\Mbar_3$ of stable genus $3$ curves. This provides a refinement of the result of Faber with a completely indipendent method. The approach is a generalization of the one used in \cite{DiLorPerVis}: we introduce an Artin stack, which is called the stack of $A_r$-stable curves, where we allow curves with $A_r$-singularities to appear. The idea is to compute the Chow ring of this newly introduced stack in the genus $3$ case and then, using localization sequence, find a description for the Chow ring of $\Mbar_3$. The stack $\Mtilde_{g,n}$ introduced in \cite{DiLorPerVis} is cointained as an open substack inside our stack. 

\subsection*{Outline of the paper}

This is the second paper in the series. It focuses on studying the moduli stack $\Htilde_g^r$ of hyperelliptic $A_r$-stable curves, which is one of the strata considered for the computation of the Chow ring of $\Mtilde_g^r$ for $g=3$. In this paper, we introduce $\Htilde_g^r$, prove that it is a smooth closed substack of $\Mtilde_g^r$ and describe it in terms of cyclic covers of twisted genus $0$ curves. This description is fundamental for the computation of the Chow ring.

Specifically, \Cref{sec:1} is dedicated to studying the possible involutions (and relative quotients) of the complete local ring of an  $A_r$-singularity. It focuses on understanding what kinds of singularities appear in the geometric quotients, whether the quotient morphism is flat and describing the fixed locus of the involution.

In \Cref{sec:2}, we introduce the moduli stack $\Htilde_g^r$ of hyperelliptic $A_r$-stable curves of genus $g$ generalizing the definition for the stable case, prove that it is a smooth stack and that it contains the stack of smooth hyperelliptic curves of genus $g$ as a dense open substack. To do this, we give an alternative description of $\Htilde_g^r$ as the moduli stack of cyclic covers of degree $2$ over twisted curves of genus $0$. This description is one of the main reasons why we choose $A_r$-singularities for our computation as they appear naturally in the case of ramified branching divisor.

 Finally, in \Cref{sec:3} we prove the main theorem of the paper, namely that the natural morphism $$\eta:\Htilde_g^r \rightarrow \Mtilde_g^r.$$  is a closed immersion of algebraic stacks as we expect from the stable case. The proof is long and uses several different tools.
 For the injectivity on geometric points, we study the combinatorics of hyperelliptic curves to prove the unicity of the hyperelliptic involution. The unramifiedness of $\eta$ follows from the unicity of the hyperelliptic involution also for infinitesimal deformations. Properness is proved by studying degenerations of families of $A_r$-curves. In particular we prove hyperelliptic involutions lift from the generic fibers of families of $A_r$-stable curves over DVR's. We want to remark that it is \emph{not} another proof of the properness of the map $\Hbar_g \rightarrow \Mbar_g$ because we use the result for the classical stable case to prove this more general result. Specifically, we only use it in \Cref{lem:not-poss-comp}. We are interested in finding a proof that does not rely on the stable case.

\section{$A_r$-singularities and involutions}\label{sec:1}

In this section we study the possible involutions acting on a singularity of type $A_r$.

Let $r$ be a nonnegative integer and $k$ be an algebraically closed field with characteristic different from $r+1$ and $2$. The complete local $k$-algebra  $$A_r:=k[[x,y]]/(y^2-x^{r+1})$$
is called an $A_r$-singularity. By definition, $A_0$ is a regular ring.		

Suppose we have an involution $\sigma$ of $A_r$. Because the normalization of a noetherian ring is universal among dominant morphism from normal rings, we know the involution lifts to the normalization of $A_r$.

\begin{remark}\label{rem:norm}
	We recall the description of the normalization of $A_r$: 
\begin{itemize}
	\item if $r$ is even, then the normalization is the morphism
	$$ \iota: A_r \arr k[[t]] $$ 
	defined by the associations $x\mapsto t^2$ and $y\mapsto t^{r+1}$;
	 
	\item if $r$ is odd, then the normalization is the morphism 
	$$ \iota: A_r \arr k[[t]]\oplus k[[t]]$$ 
	defined by the associations $x\mapsto (t,t)$ and $y \mapsto (t^{\frac{r+1}{2}}, -t^{\frac{r+1}{2}})$.
\end{itemize}

In the even case, we are identifying $A_r$ to the subalgebra of $k[[t]]$ of power series with only even degrees up to degree $r+1$. In the odd case, we are identifying $A_r$ to the subalgebra of $k[[t]]\oplus k[[t]]$ of pairs of power series which coincide up to degree $(r+1)/2$.
\end{remark}

If $\sigma$ is an involution of $k[[t]]$, we know that the differential $d\sigma$ is an involution of $k$ as a vector space over itself, therefore there exists $\xi_{\sigma} \in k$ such that $d\sigma=\xi_{\sigma} \id$ with $\xi_{\sigma}^2=1$.
We define an endomorphism $\phi_{\sigma}$ of $k[[t]]$ by the association $t\mapsto (t+\xi_{\sigma}\sigma(t))/2$. 

\begin{lemma}
	In the setting above, we have that $\phi_{\sigma}$ is an automorphism and $\sigma':=\phi_{\sigma}^{-1}\sigma\phi_{\sigma}$ is the involution of $k[[t]]$ defined by the association $t\mapsto \xi_{\sigma} t$.
\end{lemma}
\begin{proof}
	The fact that $d\phi_{\sigma}=\id$ implies $\phi_{\sigma}$ is an automorphism. The second statement follows from a straightforward computation.
\end{proof}

The idea is to prove the above lemma also for the algebra $A_r$ using the morphism $\phi_{\sigma}$. In fact we prove that in the even case the automorphism $\phi_{\sigma}$ restricts to an automorphism of $A_r$. Similarly, in the odd case we can construct an automorphism of $k[[t]]\oplus k[[t]]$, prove that it restricts to an automorphism of the subalgebra $A_r$ and describe explicitly the conjugation of the involution by this automorphism.

 \begin{proposition}\label{prop:descr-inv}
 	Every non-trivial involution of $A_r$ is one of the following: 
 	\begin{itemize}
 		\item[$(a)$] if $r$ is even, $\sigma:k[[x,y]]/(y^2-x^{r+1}) \arr k[[x,y]]/(y^2-x^{r+1})$ is defined by the associations $x\mapsto x$ and $y\mapsto -y$;	
 		\item[$(b)$] if $r$ is odd and $r\geq 3$, we get that $\sigma:k[[x,y]]/(y^2-x^{r+1}) \arr k[[x,y]]/(y^2-x^{r+1})$ is defined by one of the following associations:
 		\begin{itemize}
 			\item[$(b_1)$] $x\mapsto x$ and $ y \mapsto -y$,
 			\item[$(b_2)$] $x\mapsto -x$ and $y \mapsto -y$,
 			\item[$(b_3)$] $x\mapsto -x$ and $y \mapsto -y$;
 		\end{itemize}
 		\item[$(c)$] if $r=1$, we get that $\sigma:k[[x,y]]/(y^2-x^2) \arr k[[x,y]]/(y^2-x^2)$ is defined by one of the following associations: 
 		\begin{itemize}
 			\item[$(c_1)$] $x\mapsto x$ and $y\mapsto -y$,
 			\item[$(c_2)$] $x\mapsto -x$ and $y\mapsto -y$,
 			\item[$(c_3)$] $x\mapsto y$ and $y\mapsto x$;
 		\end{itemize}
 	\end{itemize}
 up to conjugation by an automorphism of $A_r$. 
\end{proposition}  

\begin{proof}

We start with the even case. We identify $\sigma$ with its lifting to the normalization of $A_r$ by abuse of notation. First of all, we know that $\sigma(t)=\xi_{\sigma}tp(t)$ where $p(t)\in k[[t]]$ with $p(0)=1$. Because $\sigma$ is induced by an involution on $A_r$, we have that $\sigma(t)^2=\sigma(t^2) \in A_r$. An easy computation shows that this implies $t^2p(t) \in A_r$.
We see that the images of the two generators of $A_r$ through the morphism $\phi_{\sigma}$  are inside $A_r$:
$$\phi_{\sigma}(t^2)=\frac{t^2+\sigma(t)^2 + 2\xi_{\sigma}t\sigma(t)}{4}= \frac{t^2+\sigma(t^2)+2t^2p(t)}{4} \in A_r$$
and  
$$ \phi_{\sigma}(t^{r+1})=t^{r+1}\frac{(1+p(t))^{r+1}}{2^{r+1}}\in t^{r+1}k[[t]] \subset A_r;$$
notice that if we compute the differential of the restriction $\phi_{\sigma}\vert_{A_r}:A_r \arr A_r$ we get an endomorphism of the tangent space of $A_r$ of the form 
$$
\begin{pmatrix}
  1  &  \star \\
  0  &    1
\end{pmatrix} 
$$ 
therefore  $\phi_{\sigma}\vert_{A_r}$ is an injective morphism with surjective differential between complete noetherian rings, thus it is an automorphism. Finally, if we define $\sigma':=\phi_{\sigma}^{-1}\sigma\phi_{\sigma}$, then $\sigma'\vert_{A_r}=(\phi_{\sigma}^{-1}\vert_{A_r})(\sigma\vert_{A_r})(\phi_{\sigma}\vert_{A_r})$ and we can describe its action on the generators:
\begin{itemize}
	\item $\sigma'(x)=\sigma'(t^2)=\xi_{\sigma}^2 t^2 = x$,
	\item $\sigma'(y)=\sigma'(t^{r+1})=\xi_{\sigma}^{r+1} t^{r+1} = \xi_{\sigma} y$
\end{itemize} 
 where we know that $\xi_{\sigma}^2=1$. We can have both $\xi_{\sigma}=1$ and $\xi_{\sigma}=-1$, although the first one is just the identity.

They same idea works for the odd case. We describe the lifting $\Sigma$ of the involution $\sigma$ of $A_r$ to the normalization $k[[t]]\oplus k[[t]]$. We have two possibilities: the involution $\sigma$ exchanges the two branches or not. This translates in the condition that $\Sigma$ exchanges the two connected component of the normalization (or not). Firstly, we consider the case where $\Sigma$ fixes the two connected components of the normalization, therefore we can describe $\Sigma:k[[t]]^{\oplus 2}\arr k[[t]]^{\oplus 2}$ as a matrix of the form
$$
\Sigma=
\begin{pmatrix}
      \sigma_1 & 0 \\ 0 & \sigma_2
\end{pmatrix}
$$
where $\sigma_1,\sigma_2$ are involutions of $k[[t]]$. Because $\Sigma$ is induced by an involution of $A_r$, we have that $(\sigma_1(t),\sigma_2(t))\in A_r$, i.e. $\sigma_1(t)=\sigma_2(t) \mod t^{\frac{r+1}{2}}$.

 We consider the automorphism $\Phi_{\Sigma}$ of $k[[t]]^{\oplus 2}$ which can described as a matrix of the form
$$
\Phi_{\Sigma}:=
\begin{pmatrix}
\phi_{\sigma_1} & 0 \\ 0 & \phi_{\sigma_2}
\end{pmatrix};
$$
 we have the following equalities:
$$
\Phi_{\Sigma}(t,t)=(\phi_{\sigma_1}(t),\phi_{\sigma_2}(t))=1/2(t+\xi_{\sigma_1}\sigma_1(t),t+\xi_{\sigma_2}\sigma_2(t)) \in A_r
$$ 
and 
$$
\Phi_{\Sigma}(t^{\frac{r+1}{2}},-t^{\frac{r+1}{2}})=
(\phi_{\sigma_1}(t)^{\frac{r+1}{2}},-\phi_{\sigma_2}(t)^{\frac{r+1}{2}}) \in A_r
$$
 and again we have that the differential of $\Phi_{\Sigma}\vert_{A_r}$ is of the form
$$
\begin{pmatrix}
1  &  \star \\
0  &    1
\end{pmatrix} 
$$ 
therefore $\Phi_{\sigma}\vert_{A_r}$ is an automorphism of $A_r$. Notice that if $r\geq 3$ we have that $\xi_{\sigma_1}$ and $\xi_{\sigma_2}$ are the same, but if $r=1$ we don't; nevertheless it is still true that $\Phi_{\Sigma}(t,t) \in A_1$. 
Again, we have proved that if the involution of $A_r$ fixes the two branches, then we have only a finite number of involutions up to conjugation. If $r\geq 3$ then we have only the involution described on generators by $x\mapsto \xi x$ and $y\mapsto \xi^{\frac{r+1}{2}} y$ where $\xi^2=1$. Conversely, if $r=1$ we get more possible involutions, as $\xi_{\sigma_1}$ and $\xi_{\sigma_2}$ can be different. Specifically, we get four of them; if we consider their action on the pair of generators $(x,y)$ of $A_r$, we get the following matrices describing the four involutions:
$$
\begin{pmatrix}
1  &  0 \\
0  &    1
\end{pmatrix} ,
\begin{pmatrix}
	-1  &  0 \\
	0  &   -1
\end{pmatrix},
\begin{pmatrix}
0  &  1 \\
1  &  0
\end{pmatrix},
\begin{pmatrix}
	0  &  -1 \\
	-1  &  0
\end{pmatrix};
$$
notice that the third matrix is the conjugate of the fourth one by the automorphism of $A_1$ defined on the pair of generators $(x,y)$ by the following matrix:
$$
\begin{pmatrix}
-1  &  0 \\
 0 &  1
\end{pmatrix}.
$$
Lastly, we consider the case when $\Sigma$ exchanges the connected components. Because it is an involution, $\Sigma$ is an automorphism of $k[[t]]^{\oplus 2}$ of the form 
$$
\begin{pmatrix}
0  &  \tilde{\sigma} \\
\tilde{\sigma}^{-1}  &    0
\end{pmatrix} 
$$
and because $\Sigma$ is induced by the involution $\sigma$ of $A_r$, we get that $\tilde{\sigma}$ is an involution of $k[[t]]/(t^{\frac{r+1}{2}})$, i.e. $\tilde{\sigma}^2(t)=t+t^{\frac{r+1}{2}}p(t)$ with $p(t)\in k[[t]]$. Notice that if $r\geq 3$ we get that $\tilde{\sigma}(t)=\xi_{\tilde{\sigma}}tp(t)$ where $p(t)\in k[[t]]$ with $p(0)=1$ and $\xi_{\tilde{\sigma}}^2=1$. On the contrary, if $r=1$ the previous condition is empty, i.e. $\tilde{\sigma}$ is any automorphism,. Let us first consider the case $r\geq 3$. Consider the endomorphism $\Phi_{\Sigma}$ of $k[[t]]^{\oplus 2}$ of the form 
$$
\Phi_{\Sigma}:=
\begin{pmatrix}
\phi_1  &  0 \\
0  &    \phi_2
\end{pmatrix} 
$$ 
where we define $\phi_1(t)=1/2(t+\xi_{\tilde{\sigma}}\tilde{\sigma}(t))$ and $\phi_2(t)=1/2(t+\xi_{\tilde{\sigma}}\tilde{\sigma}(t)+t^{\frac{r+1}{2}}p(t))$. Again, an easy computation shows that the automorphism $\Phi_{\Sigma}$ restricts to the algebra $A_r$ and it is in fact an automorphism. For the case $r=1$, we can simply consider $\Phi_{\Sigma}$ of the form 
$$
\Phi_{\Sigma}:=
\begin{pmatrix}
 \tilde{\sigma} &  0 \\
0  &    \id
\end{pmatrix} 
$$
which restricts to an automorphism of $A_r$ as well. As before, if $r\geq 3$ we have that the involution is of the form $x\mapsto \xi x$ and $y \mapsto -\xi^{\frac{r+1}{2}}y$. Instead, if $r=1$ we have the involution described by the association $ x \mapsto x$ and $y \mapsto -y$. 
	
\end{proof}

The previous corollary finally implies the description of the invariant we were looking for. Let us focus on the description of this invariant subalgebras (in the case of a non trivial involution). We prove the following statement.
 
\begin{corollary}
	Let $\sigma$ be a non-trivial involution of the algebra $A_r$ and let us denote by $A_r^{\sigma}$ the invariant subalgebra and by $i:A_r^{\sigma}\into A_r$ the inclusion. If we refer to the classification proved in \Cref{prop:descr-inv}, we have that
	\begin{itemize}
		\item[$(a)$] if $r$ is even, we have that $A_r^{\sigma} \simeq A_0$, the inclusion $i$ is faithfully flat and the fixed locus of $\sigma$ has length $r+1$ and it is a Cartier divisor;
		\item[$(b)$] if $r:=2k-1$ is odd and $k\geq 2$ we have that 
		\begin{itemize}
			\item[$(b_1)$] $A_r^{\sigma} \simeq A_0$, the inclusion $i$ is faithfully flat, the fixed locus of $\sigma$ has length $r+1$ and it is the support of a Cartier divisor;
			\item[$(b_2)$] $A_r^{\sigma} \simeq A_{k-1}$, the inclusion $i$ is faithfully flat, the fixed locus of $\sigma$ has length $2$ and it is the support of a Cartier divisor;
			\item[$(b_3)$] $A_r^{\sigma} \simeq A_k$, the inclusion $i$ is not flat, the fixed locus is of length $1$ and it is not the support of a Cartier divisor;
		\end{itemize}
		\item[$(c)$] if $r=1$, we have that
		\begin{itemize}
			\item[$(c_1)$] $A_1^{\sigma} \simeq A_0$, the inclusion $i$ is faithfully flat, the fixed locus of $\sigma$ has length $2$ and it is the support of a Cartier divisor;
			\item[$(c_2)$] $A_1^{\sigma} \simeq A_1$, the inclusion $i$ is not faithfully flat, the fixed locus of $\sigma$ has length $1$ and it is not the support of a Cartier divisor;
			\item[$(c_3)$] $A_1^{\sigma} \simeq A_1$, the inclusion $i$ is faithfully flat and the fixed locus coincides with one of two irreducible components.
		\end{itemize}
	\end{itemize}
\end{corollary}

\begin{proof}
Let us start with the even case. Thanks to \Cref{prop:descr-inv}, we have that the involution $\sigma$ of $A_r$ is defined by the associations $x\mapsto x$ and $y\mapsto -y$ (up to conjugation by an isomorphism of $A_r$). Therefore it is clear that $A_r^{\sigma} \simeq  k[[x]]$ and the quotient morphism is induced by the inclusion $$A_0 \simeq k[[x]] \subset \frac{k[[x,y]]}{(y^2-x^{r+1})} \simeq A_r;$$
 the same is true for the odd case when the involution $\sigma$ acts in the same way. In this case the algebras extension (corrisponding to the quotient morphism) is faithfully flat, the fixed locus is the support of a Cartier divisor defined by the ideal $(y)$ in $A_r$ and it has length $r+1$.

Now we consider the case $r=2k-1$ with $\sigma$ acting as follows: $\sigma(x)=-x$ and $\sigma(y)=y$. A straightforward computation shows that the invariant algebra $A_r^{\sigma}$ is of the type $A_{k-1}$. To be precise, the inclusion of the invariant subalgebra can be described by the morphism
$$ i: A_{k-1}\simeq\frac{k[[x,y]]}{(y^2-x^{k})} \hooklongrightarrow \frac{k[[x,y]]}{(y^2-x^{2k})}\simeq A_r$$ 
where $i(x)=x^2$ and $i(y)=y$. In this case we get the $A_r$ is a faithfully flat $A_{k-1}$-algebra and the fixed locus is the support of a Cartier divisor defined by the ideal $(x)$ in $A_r$ and it has length $2$.

If $\sigma$ is defined by the associations $x\mapsto -x$ and $y \mapsto -y$ and $r=2k-1$, then the invariant subalgebra $A_r^{\sigma}$ is of type $A_k$ and the quotient morphism is defined by the inclusion

$$ i: A_k\simeq \frac{k[[x,y]]}{(y^2-x^{k+1})} \hooklongrightarrow \frac{k[[x,y]]}{(y^2-x^{r+1})}\simeq A_r $$

where $i(x)=x^2$ and $i(y)=xy$. In contrast with the two previous cases, $A_r$ is not a flat $A_k$-algebra and the fixed locus is not (the support of) a Cartier divisor, it is in fact defined by the ideal $(x,y)$ in $A_r$ and it has length $1$. 

Finally, we consider the case where $r=2$ and the action of $\sigma$ is described by $x\mapsto y$ and $y \mapsto x$. If you consider the isomorphism 
$$ \frac{k[[u,v]]}{(uv)} \rightarrow  \frac{k[[x,y]]}{(y^2-x^2)}$$
defined by the associations $u\mapsto y+x$ and $v\mapsto y-x$  we get that the invariant subalgebra is defined by the inclusion 
$$ i: A_1\simeq \frac{k[[u,v]]}{(uv)} \hooklongrightarrow \frac{k[[u,v]]}{(uv)}$$ 
where $i(u)=u$ and $i(v)=v^2$. Notice that in this situation the algebras extension is flat but the fixed locus is not a Cartier divisor, as it is defined by $(v)$, which is a zero divisor in $A_1$.
\end{proof}

\begin{remark}
	If $r$ is odd and $r\geq 3$ (case $(b)$), we have that every involution gives a different quotient. The same is not true for the case $r=1$ as we can obtain the nodal singularity in two ways.
\end{remark}

\begin{remark}\label{rem:fix-locus}
 Notice that $(c_3)$ is the only case when the fixed locus is an irreducible component. This situation do not appear in the stack of hyperelliptic curves as we consider only involutions with finite fixed locus. 
\end{remark}
 
 We end this section with a technical lemma which will be useful afterwards. 
 \begin{lemma}\label{lem:local-node-involution}
 	Let $(R,m)\hookrightarrow (S,n)$ be a flat extension of noetherian complete local rings over $k$ such that  $$S\otimes_R R/m \simeq A_1.$$
 	Suppose we have an $R$-involution $\sigma$ of $S$ such that $\sigma \otimes R/m$ (seen as an involution of $A_1$) does not exchange the two irreducible components. Hence there exists a $R$-isomorphism 
 	$$ S \simeq R[[x,y]]/(xy-t) $$
 	where $t \in R$ such that $\sigma$ (seen as an involution of the right-hand side of the isomorphism) acts as follows: $\sigma(x)=\xi_2x$ and $\sigma(y)=\xi_1y$ for some $\xi_i \in k$ such that $\xi_i^2=1$ for $i=1,2$. Furthermore, if $\xi_1=-\xi_2$ we have $t=0$.

 \end{lemma}

 \begin{proof}
 For the sake of notation, we still denote by $\sigma$ the involution $\sigma \otimes R/m^{n+1}$. We inductively construct elements $x_n,y_n$ in $S_n:=S\otimes R/m^{n+1}$ and $t_n \in R/m^{n+1}$ such that 
 	\begin{itemize}
 		\item[1)] $\sigma(x_n)=\xi_1x_n$,
 		\item[2)] $\sigma(y_n)=\xi_2y_n$,
 		\item[3)] $x_ny_n=t_n$ in $S_n$;
 	\end{itemize}
 for some $\xi_i \in k$ indipendent of $n$ such that $\xi_i^2=1$ for $i=1,2$. The case $n=0$ follows from \Cref{prop:descr-inv} and it gives us $\xi_i$ for $i=1,2$. Suppose we have constructed $(x_n,y_n,t_n)$ with the properties listed above. 

 Consider two general liftings $x'_{n+1},y'_{n+1}$ in $S_{n+1}$. 
 
 We define $x''_{n+1}:=(x'_{n+1}+\xi_1\sigma(x'_{n+1}))/2$ and $y''_{n+1}:=(y'_{n+1}+\xi_2\sigma(y'_{n+1}))/2$. The pair $(x''_{n+1},y''_{n+1})$ clearly verify the properties 1) and 2). A priori, $x''_{n+1}y''_{n+1}=t_n+h$ with $h$ an element in $S_{n+1}$ such that its restriction in $S_n$ is zero. The flatness of the extension implies that 
 $$ \ker(S_{n+1}\rightarrow S_n)\simeq A \otimes_R (m^{n+1}/m^{n+2})\simeq S_0 \otimes_k (m^{n+1}/m^{n+2})$$ and therefore $$ h = h_0 + x''_{n+1}p(x''_{n+1}) + y''_{n+1}q(y''_{n+1})$$ 
 where all the coefficients of the polynomial $p$ and $q$ (and clearly $h_0$) are in $m^{n+1}/m^{n+2}$. if we define 
 \begin{itemize}
 	\item $t_{n+1}:=t_n+h_0$,
 	\item $x_{n+1}:=x''_{n+1}+q(y''_{n+1})$,
 	\item $y_{n+1}:=y''_{n+1}+p(x''_{n+1})$,
 \end{itemize}  
 the third condition above is verified but we have to prove that the first two are still verified for $x_{n+1}$ and $y_{n+1}$. 

 Using the fact that $\sigma(x''_{n+1}y''_{n+1})=\xi_1\xi_2x''_{n+1}y''_{n+1}$, we reduce to analyze three cases.
 
 If $\xi_1=\xi_2=1$, there is nothing to prove.
 If $\xi_1=-\xi_2$, a computation shows that   
 \begin{itemize}
 	\item $h_0=0$,
 	\item $ p\equiv 0$,
 	\item $q(y''_{n+1})=\tilde{q}(y''^2_{n+1})$
 \end{itemize}
for a suitable polynomial $\tilde{q}$ with coefficients in $m^{n+1}/m^{n+2}$. 
 If $\xi_1=\xi_2=-1$, a computation shows that   
\begin{itemize}
	\item $ p(x''_{n+1})=\tilde{p}(x''^2_{n+1})$,
	\item $q(y''_{n+1})=\tilde{q}(y''^2_{n+1})$
\end{itemize}
for a suitable $\tilde{p},\tilde{q}$ polynomials with coefficients in $m^{n+1}/m^{n+2}$. 

Hence $(x_{n+1},y_{n+1},t_{n+1})$ satisfies the conditions 1), 2) and 3),
therefore we have a morphism of flat $R/m^{n+1}$-algebras
$$ (R/m^{n+1})[[x_{n+1},y_{n+1}]]/(x_{n+1}y_{n+1}-t_{n+1}) \longrightarrow S_{n+1}$$ 
which is an isomorphism modulo $m$, therefore it is an isomorphism. If we pass to the limit we get the result.
 \end{proof}

\section{Moduli stack of $A_r$-stable hyperelliptic curves}\label{sec:2}

In this section, we briefly recall the definition of the moduli stack of $A_r$-stable curves and their moduli stack. See Section 1.2 of \cite{PerTesi} for a more detailed treatment. Moreover, we introduce the notion of hyperelliptic $A_r$-stable curve and prove that the moduli stack of hyperelliptic $A_r$-stable curves is smooth.

Fix a nonnegative integer $r$. Let $g$ be an integer with $g\geq 2$ and $n$ be a nonnegative integer.
\begin{definition}	Let $k$ be an algebraically closed field and $C/k$ be a proper reduced connected one-dimensional scheme over $k$. We say the $C$ is a \emph{$A_r$-prestable curve} if it has at most $A_r$-singularity, i.e. for every $p\in C(k)$, we have an isomorphism
		$$ \widehat{\cO}_{C,p} \simeq k[[x,y]]/(y^2-x^{h+1}) $$ 
	with $ 0\leq h\leq r$. Furthermore, we say that $C$ is $A_r$-stable if it is $A_r$-prestable and the dualizing sheaf $\omega_C$ is ample. A $n$-pointed $A_r$-stable curve over $k$ is $A_r$-prestable curve together with $n$ smooth distinct closed points $p_1,\dots,p_n$ such that $\omega_C(p_1+\dots+p_n)$ is ample.
\end{definition}
\begin{remark}
	Notice that a $A_r$-prestable curve is l.c.i by definition, therefore the dualizing complex is in fact a line bundle. 
\end{remark}

We fix a base field $\kappa$ where all the primes smaller than $r+1$ are invertible. Every time we talk about genus, we intend arithmetic genus, unless specified otherwise. We recall a useful fact.

\begin{remark}\label{rem:genus-count}
	Let $C$ be a connected, reduced, one-dimensional, proper scheme over an algebraically closed field. Let $p$ be a rational point which is a singularity of $A_r$-type. We denote by $b:\widetilde{C}\arr C$ the partial normalization at the point $p$ and by $J_b$ the conductor ideal of $b$. Then a straightforward computation shows that \begin{enumerate}
		\item if $r=2h$, then $g(C)=g(\widetilde{C})+h$;
		\item if $r=2h+1$ and $\widetilde{C}$ is connected, then $g(C)=g(\widetilde{C})+h+1$,
		\item if $r=2h+1$ and $\widetilde{C}$ is not connected, then $g(C)=g(\widetilde{C})+h$.
	\end{enumerate}
	If $\widetilde{C}$ is not connected, we say that $p$ is a separating point. Furthermore, Noether formula gives us that $b^*\omega_C \simeq \omega_{\widetilde{C}}(J_b^{\vee})$.
\end{remark}

 We can define $\Mtilde_{g,n}^r$ as the fibered category over $\kappa$-schemes  whose objects are the datum of $A_r$-stable curves over $S$ with $n$ distinct sections $p_1,\dots,p_n$ such that every geometric fiber over $S$ is a  $n$-pointed $A_r$-stable curve. These families are called \emph{$n$-pointed $A_r$-stable curves} over $S$. Morphisms are just morphisms of $n$-pointed curves.

We recall the following description of $\Mtilde_{g,n}^r$. See Theorem 1.2.7 of \cite{PerTesi} for the proof of the result. 

\begin{theorem}\label{theo:descr-quot}
	$\Mtilde_{g,n}^r$ is a smooth connected algebraic stack of finite type over $\kappa$. Furthermore, it is a quotient stack: that is, there exists a smooth quasi-projective scheme X with an action of $\GL_N$ for some positive $N$, such that 
	$ \Mtilde_{g,n}^r \simeq [X/\GL_N]$.
\end{theorem}

\begin{remark}\label{rem: max-sing}
Recall that we have an open embedding $\Mtilde_{g,n}^r  \subset \Mtilde_{g,n}^s$ for every $r\leq s$. Notice that $\Mtilde_{g,n}^r=\Mtilde_{g,n}^{2g+1}$ for every $r\geq 2g+1$. 
\end{remark}

The usual definition of the Hodge bundle extends to our setting. See Proposition 1.2.9 of \cite{PerTesi} for the proof.  As a consequence we obtain a locally free sheaf $\HH_{g}$ of rank~$g$ on $\mt_{g, n}^r$, which is called \emph{Hodge bundle}.

Now, we introduce the main actor of this paper: the moduli stack of hyperelliptic $A_r$-stable curves. 

\begin{definition}
	Let $C$ be an $A_r$-stable curve of genus $g$ over an algebraically closed field. We say that $C$ is hyperelliptic if there exists an involution $\sigma$ of $C$ such that the fixed locus of $\sigma$ is finite and the geometric categorical quotient, which is denoted by $Z$, is a reduced connected nodal curve of genus $0$. We  call the pair $(C,\sigma)$ a \emph{hyperelliptic $A_r$-stable curve} and such $\sigma$  is called a \emph{hyperelliptic involution}.
\end{definition}

Finally we can define $\Htilde_g^r$ as the following fibered category: its objects are the data of a pair $(C/S,\sigma)$ where $C/S$ is an $A_r$-stable curve over $S$ and $\sigma$ is an involution of $C$ over $S$ such that $(C_s,\sigma_s)$ is a $A_r$-stable hyperelliptic curve of genus $g$ for every geometric point $s \in S$. These are called \emph{hyperelliptic $A_r$-stable curves over $S$}. A morphism is a morphism of $A_r$-stable curves that commutes with the involutions. We clearly have a morphism of fibered categories 
$$\eta:\Htilde_g^r  \larr \Mtilde_g^r$$
 over $\kappa$ defined by forgetting the involution. This morphism is known to be a closed embedding for $r\leq 1$, where both source and target are smooth algebraic stacks of finite type over $\kappa$. We always assume that $2$ is invertible in $\kappa$.
 \begin{remark}
 	 We have a morphism $i:\Htilde_g^r \arr \mathit{I}_{\Mtilde_g^r}$ over $\Mtilde_g^r$ induced by $\eta$ where $\mathit{I}_{\Mtilde_g^r}$ is the inertia stack of $\Mtilde_g^r$. This factors through $\mathit{I}_{\Mtilde_g^r}[2]$, the two torsion elements of the inertia, which is closed inside the inertia stack. It is fully faithful by definition. This implies $\Htilde_g^r$ is a prestack in groupoid. It is easy to see that it is in fact a stack in groupoid.
 \end{remark}

We want to describe $\Htilde_g^r$ as a connected component of the $2$-torsion of the inertia of $\Mtilde_g^r$. To do so, we need to following two lemmas.  

\begin{lemma}\label{lem:quotient}
	Let $n$ be a positive integer coprime with the characteristic of $\kappa$.
	Let $X\arr S$ be a finitely presented, proper, flat morphism of schemes over $\kappa$ and let $\mmu_{n,S}\arr S$ the group of $n$-th roots acting on $X$ over $S$. Then there exists a geometric categorical quotient $X/\mmu_n\arr S$ which it is still  flat, proper and finitely presented.  
\end{lemma}
\begin{proof}
First of all, we know the existence and also the separatedness (see Corollary 5.4 of\cite{Rydh}) because $\mmu_{n,S}$ is a finite locally free group scheme over $S$. Since $\mmu_{n}$ is diagonalizable, we know that the formation of the quotient commutes with base change and that flatness is preserved. We also know that if $S$ is locally noetherian, we get that $X/\mmu_{n,S} \arr S$ is locally of finite presentation and proper (see Proposition 4.7 of \cite{Rydh}). Because proper and locally of finite presentation are two local conditions on the target, we can reduce to the case $S=\spec R$ is affine. Using now that the morphism $X\arr S$ is of finite presentation we get that there exists $R_0 \subset R$ subalgebra of $R$ which is of finite type over $\kappa$ (therefore noetherian) such that there exist a morphism $X_0\arr S_0:=\spec R_0$ proper and flat, and an action of $\mmu_{S_0,n}$ over $X_0/S_0$ such that the pullback to $S$ is our initial data. Because formation of the quotient commutes with base change, we get that we can assume $S$ noetherian and we are done.
\end{proof}

\begin{lemma}\label{lem:conn-comp}
	Let $C\arr S$ be a $A_r$-stable curve over $S$, and $\sigma$ be an automorphism of the curve over $S$ such that $\sigma^2=\id$. Then there exists an open and closed subscheme $S'\subset S$ such that the following holds: a morphism $f:T\arr S$  factors through $S'$ if and only if $(C\times_S T, \sigma \times_S T)$ is a hyperelliptic $A_r$-stable curve over $T$.
\end{lemma}
  
\begin{proof}
Consider $$S':=\{s\in S| (C_s,\sigma_s) \in \Htilde_g^r(s)  \}\subset S$$
i.e. the subset where the geometric fibers over $S$ are hyperelliptic $A_r$-stable curves of genus $g$. If we prove $S'$ is open and closed, we are done.
Let $Z\rightarrow S$ be the geometric quotient of $C$ by the involution. Because of \Cref{lem:quotient}, $Z\rightarrow S$ is flat, proper and finitely presented and thus both the dimension and the genus of the geometric fibers over $S$ are locally costant function on $S$. It remain to prove that the finiteness of the fixed locus is an open and closed condition. Notice that in general the fixed locus is not flat over $S$, therefore this is not trivial.

The openess follows from the fact the fixed locus of the involution is proper over $S$, and therefore we can use the semicontinuity of the dimension of the fibers. 
 
Regarding the closedness, the fixed locus of the geometric fiber over a point $s\in S$ is positive dimensional only when we have a projective line on the fiber $C_s$ such that intersects the rest of the curve only in separating nodes (because of \Cref{prop:descr-inv}). Therefore the result is a direct conseguence of \Cref{lem:local-node-involution}.
\end{proof}
 	
\begin{proposition}\label{prop:open-closed-imm}
	The morphism $i:\Htilde_g^r \arr \mathit{I}_{\Mtilde_g^r}[2]$ is an open and closed immersion of algebraic stacks. In particular, $\Htilde_g^r$ is a closed substack of $\mathit{I}_{\Mtilde_g^r}$ and it is locally of finite type (over $\kappa$). 
\end{proposition}

\begin{proof}
We first prove that $\Htilde_g^r$ is an algebraic stack, and an open and closed substack of $\mathit{I}_{\Mtilde_g^r}[2]$.
First of all, we need to prove that the diagonal of $\Htilde_g^r$ is representable by algebraic spaces. It follows from the following fact: given a morphism of fibered categories $X\arr Y$, we can consider the $2$-commutative diagram
$$
\begin{tikzcd}
X \arrow[d, "\Delta_X"] \arrow[r, "f"] & Y \arrow[d, "\Delta_Y"] \\
X\times X \arrow[r, "{(f,f)}"]       & Y\times Y;            
\end{tikzcd}
$$
 if $f$ is fully faithful, then the diagram is also cartesian.
Secondly, we need to prove that the morphism $i$ is representable by algebraic spaces and it is an open and closed immersion. Suppose we have a morphism $ V\arr \mathit{I}_{\Mtilde_g^r}$ from a $\kappa$-scheme $S$. Thus we have a cartesian diagram:
 $$
 \begin{tikzcd}
 F \arrow[d] \arrow[r, "i_S"] & S \arrow[d]              \\
 \Htilde_g^r \arrow[r, "i"]   & \mathit{I}_{\Mtilde_g^r}
 \end{tikzcd}
 $$
 where $F$ is equivalent to a category fibered in sets as $i$ is fully faithful. We can describe $F$ in the following way: for every $T$ a $\kappa$-scheme, we have $$F(T)=\{ f:T\arr S |(C_S\times_S T,\sigma_S \times_S T) \in \Htilde_g^r(T) \}$$
 where $(C_S,\sigma_S) \in \mathit{I}_{\Mtilde_g^r}(S)$ is the object associated to the morphism $S\arr   \mathit{I}_{\Mtilde_g^r}$. By \Cref{lem:conn-comp}, we deduce $F=S'$ and the morphism $i_S$ is an open and closed immersion.
\end{proof}

In the last part of this section we introduce another description of $\Htilde_g^r$, useful for understanding the link with the smooth case. We refer to \cite{AbOlVis} for the theory of twisted nodal curves, although we consider only twisted curves with $\mu_2$ as stabilizers and with no markings. 

The first description is a way of getting cyclic covers from $A_r$-stable hyperelliptic curves. 

\begin{definition}
	Let $\cC(2,g,r)'$ be the category fibered in groupois whose object are morphisms $f:C\rightarrow \cZ$ over some base scheme $S$ such that $C\rightarrow S$ is a family of $A_r$-stable genus $g$ curve, $\cZ\rightarrow S$ is a family of twisted curves of genus $0$ and $f$ is a finite flat morphism of degree $2$ which is generically \'etale. Morphisms are commutative diagrams of the form
	$$
	\begin{tikzcd}
	C \arrow[d, "\phi_C"] \arrow[r, "f"] & \cZ \arrow[d, "\phi_Z"] \arrow[r] & S \arrow[d] \\
	C' \arrow[r, "f'"]                   & \cZ' \arrow[r]                    & S'.         
	\end{tikzcd}
	$$
\end{definition}

\begin{remark}
	The definition implies that the morphism $f$ is \'etale over the stacky locus of $\cZ$ . First of all, we can reduce to the case of $S$ being a spectrum of an algebraically closed field. Let $\xi: B\mu_2\hookrightarrow \cZ$ be a stacky node of $\cZ$, thus $f^{-1}(\xi)\rightarrow B\mu_2$ is finite flat of degree two. It is clear that  $f^{-1}(\xi)\subset C$ implies $f^{-1}(\xi)\rightarrow B\mu_2$ \'etale. It is easy to prove that over a stacky node of $\cZ$ there can only be a separating node of $C$. 
\end{remark}

The theory of cyclic covers (see for instance \cite{ArVis}) guarantees the existence of the functor of fibered categories 
$$\gamma:\cC(2,g,r)' \longrightarrow \Htilde_g^r$$
defined in the following way on objects: if $f:C\rightarrow Z$ is finite flat of degree $2$, we can give a $\ZZ/(2)$-grading on $f_*\cO_C$, because it splits as the sum of $\cO_{\cZ}$ and some line bundle $\cL$ on $\cZ$ with a section $\cL^{\otimes 2}\hookrightarrow \cO_{\cZ}$. This grading defines an action of $\mu_2$ over $C$. Everything is relative to a base scheme $S$. The geometric quotient by this action is the coarse moduli space of $\cZ$, which is a genus $0$ curve. The fact that $f$ is generically \'etale, implies that the fixed locus is finite. 

We prove a general lemma which gives us the uniqueness of an involution once we fix the geometric quotient. 

\begin{lemma}\label{lem:unique-inv-quotient}
	Suppose $S$ is a $\kappa$-scheme.  Let $X/S$ be a finitely presented, proper, flat $S$-scheme and $\sigma_1,\sigma_2$ be two involutions of $X/S$. Consider a geometric quotient $\pi_i:X\rightarrow Y_i$ of the involution $\sigma_i$ for $i=1,2$. If there exists an isomorphism  $\psi:Y_1 \rightarrow Y_2$ of $S$-schemes which commutes with the quotient maps, then $\sigma_1=\sigma_2$. 
\end{lemma}
\begin{proof}
	Let $T$ be a $S$-scheme and $t:T\rightarrow X$ be a $T$-point of $X$. We want to prove that $\sigma_1(t)=\sigma_2(t)$.
	Fix $i\in \{1,2\}$ and consider the cartesian diagram
	$$
	\begin{tikzcd}
	X_{\pi_i(t)} \arrow[d] \arrow[r] & X \arrow[d, "\pi_i"] \\
	T \arrow[r, "\pi_i(t)"]          & Y_i,                
	\end{tikzcd}
	$$
	thus we have that there are at most two sections of the morphism $X_{\pi_i(t)}\rightarrow T$, namely $t$ and $\sigma_i(t)$. Using the fact that $\pi_2=\psi\circ \pi_1$, it follows easily that $\sigma_1(t)=\sigma_2(t)$.
\end{proof}	

\begin{proposition}\label{prop:cyclic-covers}
	The functor $\gamma:\cC(2,g,r)' \rightarrow \Htilde_g^r$ is an equivalence of fibered categories.
\end{proposition}

\begin{proof}
	We explicitly construct an inverse. Let $(C,\sigma)$ be a hyperelliptic $A_r$-stable curve of genus $g$ over $S$. Consider $F_{\sigma}\subset C$ the fixed locus of $\sigma$, which is finite over $S$. \Cref{prop:descr-inv} implies that the defining ideal of $F_{\sigma}$ in $C$ is locally generated by at most $2$ elements and the quotient morphism is not flat exactly in the locus where it is generated by $2$ elements. By the theory of Fitting ideals, we can describe the non-flat locus as the vanishing locus of the $1$-st Fitting ideal of $F_{\sigma}$. Let $N_{\sigma}$ be such locus in $C$. \Cref{lem:local-node-involution} implies that $N_{\sigma}$ is also open inside $F_{\sigma}$, therefore $F_{\sigma}\setminus N_{\sigma}$ is closed inside $C$. If we look at the stacky structure on the image of $F_{\sigma}\setminus N_{\sigma}$ through the stacky quotient morphism $C\rightarrow [C/\sigma]$, we know that the stabilizers are isomorphic to $\mu_2$ as it is contained in the fixed locus. We denote by $[C/\sigma] \rightarrow \cZ$ the rigidification along $F_{\sigma}\setminus N_{\sigma}$ of $[C/\sigma]$. Because we are dealing with linearly reductive stabilizers, we know that the rigidification is functorial and $\cZ\rightarrow S$ is still flat (proper and finitely presented).
	
	We claim that the composition $$C\longrightarrow [C/\sigma] \longrightarrow \cZ$$ is an object of $\cC(2,g,r)$. As the quotient morphism is \'etale, the only points where we need to prove flatness are the ones in $F_{\sigma}\setminus N_{\sigma}$, which are fixed points. Because the morphism $C\rightarrow \cZ$ locally at a point $p\in F_{\sigma}\setminus N_{\sigma}$ is the same as the geometric quotient, it follows from \Cref{prop:descr-inv} that it is flat. Thus $C\rightarrow \cZ$ is a finite flat morphism of degree $2$, generically \'etale as $F_{\sigma}$ is finite. 
	
	Hence, this construction defines a functor 
	$$\tau: \Htilde_g^r \longrightarrow \cC(2,g,r)'$$
	where the association on morphisms is defined in the natural way. A direct inspection using \Cref{lem:unique-inv-quotient} shows that $\tau$ and $\gamma$ are one the quasi-inverse of the other. 
	
\end{proof}

Using the theory developed in \cite{ArVis}, we know that $\cC(2,g,r)'$ is isomorphic to a stack of cyclic covers, namely the datum $C\rightarrow \cZ$ is equivalent to the triplet $(\cZ,\cL,i)$ where $\cZ$ is a twisted nodal curve, $\cL$ is a line bundle over $\cZ$ and $i:\cL^{\otimes 2}\rightarrow \cO_{\cZ}$ is a morphism of $\cO_{\cZ}$-modules. The vanishing locus of $i^{\vee}$ determines a subscheme of $\cZ$ which consists of the branching points of the cyclic cover. 

To recover $C$, we consider the sheaf of $\cO_{\cZ}$-algebras $\cA:=\cO_{\cZ}\oplus \cL$ where the algebra structure is defined by the section $i:\cL^{\otimes 2}\hookrightarrow \cO_{\cZ}$. Thus we define $C:=\spec_{\cZ}(\cA)$. Clearly not every triplet $(\cZ,\cL,i)$ gives us a $A_r$-stable genus $g$ curve $C$. We need to understand what are the conditions for $\cL$ and $i$ such that $C$ is a $A_r$-stable curve of genus $g$. Because $C\rightarrow S$ is flat, proper and of finite presentation, it is a family of $A_r$-stable curve of genus $g$ if and only if the geometric fiber $C_{s}$ is a $A_r$-stable curve of genus $g$ for every point $s \in S$. Therefore we can reduce to understand the conditions for $\cL$ and $i$ over an algebraically closed field $k$. 

\begin{lemma}
	In the situation above, for every triplet $(\cZ,\cL,i)$ we have that $C$ is a proper one-dimensional Deligne-Mumford stack over $k$. Furthermore:
	\begin{itemize}
		\item[(i)] $C$ is a scheme if and only if the morphism $\cZ \rightarrow B\GG_m$ induced by $\cL$ is representable and $i$ does not vanish on the stacky nodes, 
		\item[(ii)] $C$ is reduced if and only if $i$ is injective,
		\item[(iii)] $C$ has arithmetic genus $g$ if and only if $\chi(\cL)=-g$.
	\end{itemize}
\end{lemma}

\begin{proof}
	 Because $C\rightarrow \cZ$ is finite, then $C$ is a proper one-dimensional Deligne-Mumford stack. To prove $(i)$, it is enough to prove that $C$ is an algebraic space, or equivalently it has trivial stabilizers. Because the map $f: C\rightarrow \cZ$ is affine, it is enough to check the fiber of the stacky points of $\cZ$. By the local description of the morphism, it is clear that if $p:B\mu_2 \hookrightarrow \cZ$ is a stacky point of $\cZ$, we have that the fiber $f^{-1}(p)$ is the quotient of the artinian algebra $k[x]/(x^2-h)$ by $\mu_2$, where $h \in k$. The usual description of the cyclic cover implies that $x$ is the local generator of $\cL$ and $h$ is the value of $i^{\vee}$ at the point $p$. The representability of the fiber then is equivalent to $h\neq 0$ and to the $\mu_2$-action being not trivial on $x$. Thus (i) follows.
	 	
	 Notice that $\cZ$ is clearly a Cohen-Macaulay stack and $f$ is finite flat and representable, therefore $C$ is also Cohen-Macaulay. This implies that $C$ is reduced if and only if it is generically reduced, i.e. the local ring of every irreducible component is a field. It is easy to see that this is equivalent to the fact that $i$ does not vanish in the generic point of any component.
	 	
	 As far as (iii) is concerned, firstly we prove that $\chi(\cL)\leq 0$ implies $C$ is connected. Suppose $C$ is the disjoint union of $C_1$ and $C_2$, then the involution has to send at least one irreducible component of $C_1$ to one irreducible component of $C_2$, otherwise the quotient $\cZ$ is not connected. But this implies that the restriction $f\vert_{C_1}$ is a finite flat representable morphism of degree $1$, therefore an isomorphism. In particular $f$ is a trivial \'etale cover of $\cZ$ and therefore $\chi(\cO_C)=2$, or equivalently $\chi(\cL)=1$. A straightforward computation now shows the equivalence.
	 	
\end{proof}
 
 \begin{remark}
 	Notice that $C$ is disconnected if and only if $\chi(\cO_C)=2$. We say that in this case $C$ has genus $-1$ and it is equivalent to $\chi(\cL)=1$.
 \end{remark}

Secondly, we have to take care of the $A_r$-prestable condition which makes sense only in the case when $C$ is a scheme. Therefore, we suppose $i^{\vee}$ does not vanish on the stacky points.

The $A_r$-prestable condition is encoded in the section $i^{\vee}\in \H^0(\cZ,\cL^{\otimes -2})$. The local description of the cyclic covers of degree $2$ (see \Cref{prop:descr-inv}) implies the following lemma.

\begin{lemma}
	In the situation above, $C$ is an $A_r$-stable curves if and only if $i^{\vee}$ has the following properties: 
	\begin{itemize}
		\item if it vanishes at a non-stacky node, then $r\geq 3$, $i^{\vee}$ vanishes with order $2$ and, locally at the node, it is not a zero divisor (there is a tacnode over the regular node),
		\item if it vanishes at a smooth point, than it vanishes with order at most $r$.
	\end{itemize}
\end{lemma}

Lastly, we want to understand how to describe the stability condition. For this, we need a lemma.

\begin{lemma}
	Let $(C,\sigma)$ be an $A_r$-prestable hyperelliptic curve of genus $g\geq 2$ over an algebraically closed field such that the geometric quotient $Z:=C/\sigma$ is integral of genus $0$, i.e. a projective line. Then $(C,\sigma)$ is $A_r$-stable. 
	
	Furthermore, suppose instead that $g=1$ and let $p_1,p_2$ be two smooth points such that $\sigma(p_1)=p_2$. Then $(C,p_1,p_2)$ is stable.
\end{lemma}
\begin{proof}
	If the curve $C$ is integral, there is nothing to prove. Suppose $C$ is not, then it has two irreducible components $C_1$ and $C_2$ of genus $0$ that has to be exchanged by the involution and their intersection $C_1 \cap C_2$ is a disjoint union of $A_{2k+1}$-singularities for $2k+1\leq r$. Let $p_1,\dots,p_h$ be the support of the intersection and $k_1,\dots,k_h$ integers such that $p_i$ is a $A_{2k_i+1}$-singularity for $i=1,\dots,h$. By \Cref{rem:genus-count}, we have that 
	$$  k_1+\dots+k_h-1=g\geq 2$$ 
	but at the same time Lemma 1.12 of \cite{Cat} implies
	$$\deg\omega_C\vert_{C_j}=-2+k_1+\dots+k_h=g-1>0$$
    for $j=1,2$. 
    
    Suppose now $g=1$. If $C$ is integral, then there is nothing to prove. If $C$ is not integral, again it has two irreducible components of genus $0$ such that their intersection is either two nodes or a tacnode, because of \Cref{rem:genus-count} (see the proof of \Cref{lem:genus1} for a more detailed discussion). Then again the statement follows.
\end{proof}

\begin{remark}
	In the previous lemma, we can take $p:=p_1=p_2$ to be a fixed smooth point of the hyperelliptic involution. In this case, $C$ has to be integral and therefore $(C,p)$ is $A_r$-stable. 
\end{remark}

The stability condition makes sense when $C$ is Gorestein, and in particular when it is $A_r$-prestable. Therefore suppose we are in the situation when $C$ is a $A_r$-prestable curve. We translate the stability condition on $C$ to a condition on the restrictions of $i$ to the irreducible components of $\cZ$.

Given an irreducible component $\Gamma$ of $\cZ$, we can define a quantity $$g_{\Gamma}:=\frac{n_{\Gamma}}{2}-\deg \cL\vert_{\Gamma}-1$$ where $n_{\Gamma}$ is the number of stacky points of $\Gamma$. It is easy to see that  $g_{\Gamma}$ coincides with the arithmetic genus of the preimage $C_{\Gamma}:=f^{-1}(\Gamma)$, when it is connected. The previous lemma implies that $\omega_{C}\vert_{C_{\Gamma}}$ is ample for all the components $\Gamma$ such that $g_{\Gamma}\geq 1$ . 

Let us try to understand the stability condition for $g_{\Gamma}=0$, i.e. $\deg\cL\vert_{\Gamma}=n_{\Gamma}/2-1$. Let $m_{\Gamma}$ be the number of points of the intersection of $\Gamma$ with the rest of the curve $\cZ$, or equivalently the number of nodes (stacky or not) lying on the component. Then the stability condition translates to $2m_{\Gamma}-n_{\Gamma}\geq 3$, because the fiber of the morphism $C\rightarrow \cZ$ of every non-stacky node of $\cZ$ is of length $2$ (either two disjoint nodes or a tacnode) in $C$ while the fiber of a stacky node ($B\mu_2 \hookrightarrow \cZ$) has length $1$.

Suppose now that the curve $C_{\Gamma}$ is disconnected. Thus $C_{\Gamma}$ is the disjoint union of two projective line with an involution that exchange them. In this case $\cL\vert_{\Gamma}$ is trivial but also $n_{\Gamma}=0$. Stability condition on $C$ is equivalent to $m_{\Gamma}\geq 3$.  Notice that $g_{\Gamma}=-1$ if and only if $C_{\Gamma}$ is disconnected, or equivalently it is the \'etale trivial cover of projective line. 

This motivates the following definition.

\begin{definition}\label{def:hyp-A_r}
	Let $\cZ$ be a twisted nodal curve over an algebraically closed field. We denote by $n_{\Gamma}$ the number of stacky points of $\Gamma$ and by $m_{\Gamma}$ the number of intersections of $\Gamma$ with the rest of the curve for every $\Gamma$ irreducible component of $\cZ$. Let $\cL$ be a line bundle on $\cZ$ and $i:\cL^{\otimes 2} \rightarrow \cO_{\cZ}$ be a morphism of $\cO_{\cZ}$-modules.  We denote by $g_{\Gamma}$ the quantity $n_{\Gamma}/2-1-\deg\cL\vert_{\Gamma}$. 
	\begin{itemize}
	\item[(a)] We say that $(\cL,i)$ is hyperelliptic if the following are true:
			\begin{itemize}
				\item[(a1)] the morphism $\cZ \rightarrow B\GG_m$ induced by $\cL$ is representable,
				\item[(a2)] $i^{\vee}$ does not vanish restricted to any stacky point.
			\end{itemize}
	\item[(b)] We say that $(\cL,i)$ is $A_r$-prestable and hyperelliptic of genus $g$ if $(\cL,i)$ is hyperelliptic, $\chi(\cL)=-g$ and the following are true:
		\begin{itemize}
			\item[(b1)] $i^{\vee}$ does not vanish restricted to any irreducible component of $\cZ$ or equivalently the morphism $i:\cL^{\otimes 2 }\rightarrow \cO_{\cZ}$ is injective,
			\item[(b2)] if $p$ is a non-stacky node and $i^{\vee}$ vanishes at $p$, then $r\geq 3$ and the vanishing locus $\VV(i^{\vee})_p$ of $i^{\vee}$ localized at $p$ is a Cartier divisor of length $2$;
			\item[(b3)] if $p$ is a smooth point and $i^{\vee}$ vanishes at $p$, then the vanishing locus $\VV(i^{\vee})_p$ of $i^{\vee}$ localized at $p$ has length at most $r+1$.
		\end{itemize}	
	\item[(c)] We say that $(\cL,i)$ is $A_r$-stable and hyperelliptic of genus $g$ if it is $A_r$-prestable and hyperelliptic of genus $g$ and the following are true for every irreducible component $\Gamma$ in $\cZ$:
			\begin{itemize}
				\item[(c1)] if $g_{\Gamma}=0$ then we have $2m_{\Gamma}-n_{\Gamma}\geq 3$,
				\item[(c2)] if $g_{\Gamma}=-1$ then we have
				$m_{\Gamma}\geq 3$ ($n_{\Gamma}=0$).
			\end{itemize}
\end{itemize}
\end{definition}

Let us define now the stack classifying these data. We denote by $\cC(2,g,r)$ the fibered category defined in the following way: the objects are triplets $(\cZ\rightarrow S,\cL,i)$ where $\cZ \rightarrow S$ is a family of twisted curves of genus $0$, $\cL$ is a line bundle of $\cZ$ and $i:\cL^{\otimes 2}\rightarrow \cO_{\cZ}$ is a morphism of $\cO_{\cZ}$-modules such that the restrictiond $(\cL_s,i_s)$ to the geometric fibers over $S$ are $A_r$-stable and hyperelliptic of genus $g$.  Morphisms are defined as in \cite{ArVis}. We have proven the following equivalence.

\begin{proposition}\label{prop:descr-hyper}
	The fibered category $\cC(2,g,r)$ is isomorphic to $\cC(2,g,r)'$.
\end{proposition}

We can use this description to get the smoothness of $\Htilde_g^r$. Firstly, we need to understand what kind of line bundles $\cL$ on $\cZ$ can appear in $\cC(2,g,r)$.

\begin{lemma}\label{lem:hyp-line-bun}
	Let $Z$ be a nodal curve of genus $0$ over an algebraically closed field $k/\kappa$, $\cL$ a line bundle on $Z$ and $s \in \H^0(Z,\cL)$. We consider the following assertions:
	\begin{itemize}
		\item[$(i)$] $s$ does not vanish identically on any irreducible component of $Z$,
		\item[$(ii)$] $\cL$ is globally generated,
		\item[$(ii')$] $\deg \cL\vert_{\Gamma}\geq 0$ for every $\Gamma$ irreducible component of $Z$,
		\item[$(iii)$] $\H^1(Z,\cL)=0$;
	\end{itemize}
	then we have $(i)\implies (ii)\iff (ii')\implies (iii)$.
\end{lemma}

\begin{proof}
	It is easy to prove that $(i) \implies(ii')$ and $(ii)\implies (ii')$. To prove that $(ii')$ implies $(ii)$ we proceed by induction on the number of components of $Z$. Let $p$ a smooth point on $Z$ and $\Gamma_p$ the irreducible component which cointains $p$. Then the morphism $$\H^0(\Gamma_p,\cL\vert_{\Gamma_p})\longrightarrow k(p)$$ is clearly surjective because $\deg \cL\vert_{\Gamma_p}\geq 0$ and $\Gamma \simeq \PP^1$. Therefore it is enough to extend a section from $\Gamma_p$ to a section to the whole $Z$. 
	
	Because the dual graph of $Z$ is a tree, we know that $Z$ can be obtained by gluing a finite number of genus $0$ curve $Z_i$ to $\Gamma_p$ such that these curves are disjoint. Hence it is enough to find a section for every $Z_i$ such that it glues in the point of intersection with $\Gamma_p$. Because everyone of the $Z_i$'s has fewer irreducible components than $Z$, we are done. The case when $p$ is singular is left as an exercise to the reader.  
	
	Finally, we need to prove that $(ii')\implies (iii)$. We know there exists a decomposition $Z=Z_1\cup Z_2$ where $Z_1$ and $Z_2$ are nodal genus $0$ curves such that $Z_1\cap Z_2$ has length $1$. Let $i_h:Z_h \hookrightarrow Z$ the closed embedding for $h=1,2$. Thus if we can consider the exact sequence of vector spaces 
	$$ 0\rightarrow \H^0(\cL)\rightarrow \H^0(i_1^*\cL)\oplus\H^0(i_2^*\cL) \rightarrow k \rightarrow \H^1(\cL) \rightarrow \H^1(i_1^*\cL)\oplus \H^1(i_2^*\cL) \rightarrow 0;$$  
	it is clear that the morphism $\H^0(i_1^*\cL)\oplus\H^0(i_2^*\cL) \rightarrow k$ is surjective (both $i_1^*\cL$ and $i_2^*\cL$ are globally generated) and therefore $\H^1(Z,\cL)=\H^1(Z_1,i_1^*\cL)\oplus \H^1(Z_2,i_2^*\cL)$. We get that 
	$$\H^1(Z,\cL)=\bigoplus_{\Gamma} \H^1(\Gamma, \cL\vert_{\Gamma})$$
	where the sum is indexed over the irreducible components $\Gamma$ of $Z$. Thus it is enough to prove that $\H^1(\Gamma,\cL\vert_{\Gamma})=0$ for every $\Gamma$ irreducible component of $Z$, which follows from $(iii))$.
\end{proof} 

\begin{proposition}\label{prop:smooth-hyp}
	The moduli stack $\Htilde_g^r$ of $A_r$-stable hyperelliptic curves of genus $g$ is smooth and the open $\cH_g$ parametrizing smooth hyperelliptic curves is dense in $\Htilde_g^r$. In particular $\Htilde_g^r$ is connected.
\end{proposition}

\begin{proof}
We use the description of $\Htilde_g^r$ as $\cC(2,g,r)$. First of all, let us denote by $\cP$ the moduli stack parametrizing the pair $(\cZ,\cL)$ where $\cZ$ is a twisted curve of genus $0$ and $\cL$ is a line bundle on $\cZ$. Consider the natural morphism $\cP \rightarrow \cM_0^{\rm tw}$ defined by the association $(\cZ,\cL)\mapsto \cZ$, where $\cM_0^{\rm tw}$ is the moduli stack of twisted curve of genus $0$ (see Section 4.1 of \cite{AbGrVis}).Proposition 2.7 \cite{AbOlVis} implies the  formal smoothness of this morphism. We have the smoothness of $\cM_0^{\rm tw}$ thanks to Theorem A.6 of \cite{AbOlVis}. Therefore $\cP$ is formally smooth over the base field $\kappa$.

 Let us consider now $\cP_g^0$, the substack of $\cP_g$ whose geometric objects are pairs $(\cZ,\cL)$ such that $\chi(\cL)=-g$ and $\H^1(\cZ,\cL^{\otimes -2})=0$. To be precise, its objects consist of families of twisted curves $\cZ\rightarrow S$ and a line bundle $\cL$ on $\cZ$ such that $\H^1(\cZ_s,\cL_s^{\otimes -2})=0$ for every $s \in S$. Because the Euler characteristic $\chi$ is locally constant for families of line bundles, the semicontinuity of the $\H^1$  implies that $\cP_g^0$ is an open inside $\cP$, therefore it is formally smooth over $\kappa$.

 We have a morphism 
 $$ \Htilde_g^r\simeq \cC(2,g,r) \rightarrow \cP$$ 
 defined by the association $(\cZ,\cL,i) \mapsto (\cZ,\cL)$. This factors through $\cP_g^0$ because of \Cref{lem:hyp-line-bun}. Consider the universal object $(\pi:\cZ_{\cP}\rightarrow \cP_g^0,\cL_{\cP})$ over $\cP_g^0$. Then $\cL_{\cP}^{\otimes -2}$ satisfies base change by construction and therefore we have that $\cC(2,g,r)$ is a substack of $\VV(\pi_*\cL_{\cP}^{\otimes -2})$, the geometric vector bundle over $\cP_g^0$ associated to $\pi_*\cL_{\cP}^{\otimes -2}$. The inclusion $\cC(2,g,r)\subset \VV(\pi_*\cL_{\cP}^{\otimes -2})$ is an open immersion because of the deformation theory of $A_r$-prestable curves, which implies the smoothness of $\cC(2,g,r)$. 
 
 Given a twisted curve $\cZ$ over an algebraically closed field $k$, we can construct a family $\cZ_R\rightarrow \spec R$ of twisted curves where $R$ is a DVR such that the special fiber is $\cZ$ and the generic fiber is smooth. The smoothness of the morphism $\VV(\pi_{*}\cL^{\otimes -2})\rightarrow \Mtilde_0^{\rm tw}$ implies that given any object $(\cZ,\cL,i) \in \cC(2,g,r)'(k)$, we can lift it to $(\cZ_R,\cL_R,i_R) \in \cC(2,g,r)'(R)$ where $R$ is a DVR such that it restricts to $(\cZ,\cL,i)$ in the special fiber and such that the generic fiber is isomorphic $(\PP^1, \cO(-g-1), i)$ with $i^{\vee} \in \H^0(\PP^1,\cO(2g+2))$. Finally, the open substack of $\VV(\pi_*\cL_{\cP}^{\otimes -2})$ parametrizing sections $i^{\vee}$ without multiple roots is dense, therefore we get that we can deform every datum $(\cZ,\cL,i)$ to the datum of $(\PP^1, \cO(-g-1), i)$ where $i^{\vee}$ does not have multiple roots, which corresponds to a smooth hyperelliptic curve.
\end{proof}

\begin{remark}
	In the proof of \Cref{prop:smooth-hyp}, we have used the implication $(i) \implies (iii)$ as in  \Cref{lem:hyp-line-bun} for a twisted curve $\cZ$ of genus $0$. In fact, we can apply \Cref{lem:hyp-line-bun} to the coarse moduli space $Z$ of $\cZ$ and get the implication for $\cZ$ because the line bundle  $\cL^{\otimes -2}$ descends to $Z$ and the morphism $\cZ\rightarrow Z$ is cohomologically affine.
\end{remark}

\section{$\eta$ is a closed immersion}\label{sec:3}

We have introduced a morphism $\eta:\Htilde_g^r \arr \Mtilde_g^r$ between smooth algebraic stacks. We know that it is a closed immersion if $r\leq 1$. The rest of the paper is dedicated to proving the same result for any $r$. We are going to prove that $\eta$ is representable, formally unramified, injective on geometric points and universally closed which implies that $\eta$ is a closed immersion.

\begin{remark}
The morphism $\eta$ is faithful, as the automorphisms of $(C,\sigma)$ are by definition a subset of the ones of $C$ over any $\kappa$-scheme $S$. This implies that $\eta$ is representable.
\end{remark}

\subsection*{Injectivity on geometric points}
Firstly, we discuss why $\eta$ is injective on geometric points. We just need to prove that for every geometric point the morphism $\eta$ is full, i.e. every automorphism of a $A_r$-stable curve which is also hyperelliptic have to commute with the involution. This follows directly from the unicity of the hyperelliptic involution. Therefore our next goal is to prove that the hyperelliptic involution is unique over an algebraically closed field. We start by describing the possible quotients by an involution.

\begin{proposition}\label{prop:description-quotient}
	Let $k/\kappa$ be an algebraically closed field and $(C,\sigma)$ be a hyperelliptic $A_r$-stable curve of genus $g$. Denote by $Z$ the geometric quotient by the involution and suppose $Z$ is a reduced nodal curve of genus $0$ (with only separating nodes). Furthermore, let $c\in C$ be a closed point, $z \in Z$ be the image of $c$ in the quotient and $C_z$ be the schematic fiber of $z$, which is the spectrum of an artinian $k$-algebra. Then
	
	\begin{itemize}
		\item if $z$ is a smooth point, either 
		\begin{itemize}
			\item[(s1)] $C_z$ is disconnected and supported on two smooth points of $C$ (i.e. the quotient morphism is \'etale at $c$),
			\item[(s2)] or $C_z$ is connected and supported on a possibly singular point of $C$ (i.e. the quotient morphism is flat and ramified at $c$); 
		\end{itemize}
		\item if $z$ is a separating node, either
		\begin{itemize}
			\item[(n1)] $C_z$ is disconnected and supported on two nodes (i.e. the quotient morphism is \'etale at $c$),
			\item[(n2)] or $C_z$ is connected and supported on a tacnode (i.e. the quotient morphism is flat and ramified at $c$),
			\item[(n3)] or $C_z$ is connected and supported on a node (i.e. the quotient morphism is ramified at $c$ but not flat). 
		\end{itemize}
	\end{itemize} 

Finally, the quotient morphism is finite of generic degree $2$ and ${\rm (n3)}$ is the only case when the length of $C_z$ is not $2$ but $3$. 
\end{proposition}

\begin{proof}
	As we are dealing with the quotient by an involution, the quotient morphism is either \'etale at closed point $c\in C$ or the point $c\in C$ is fixed by the involution. If the point is in the fixed locus, we can pass to the completion and apply  \Cref{prop:descr-inv}. 
\end{proof}
\begin{remark}
	We claim that if $C$ is an $A_r$-prestable curve of genus $g$ and $\sigma$ is any involution, then the geometric quotient $Z$ is automatically an $A_r$-prestable curve. The quotient is still reduced because $C$ is reduced. We have connectedness because of the surjectivity of the quotient morphism. The description of the quotient singularities in the first section implies the claim. Therefore if the quotient $Z$ has genus $0$, then it is a nodal curve 
	(with only separating nodes).
\end{remark}
\begin{remark}
      If $z$ is a node, we can say more about $C_z$. In fact, it is easy to prove that $C_z$ is a separating closed subset, i.e. the partial normalization of $C$ along the support of $C_z$ is not connected. This follows from the fact that every node in $Z$ is separating.
\end{remark}
We start by proving the uniqueness of the hyperelliptic involution for the case of an integral curve. For the remaining part of the section, $k$ is an algebraically closed field over $\kappa$. 

\begin{proposition}\label{prop:integral}
	Let $C$ be a $A_r$-stable integral curve of genus $g\geq 2$ over $k$ and suppose are given two hyperelliptic involution $\sigma_1,\sigma_2$ of $C$. Then $\sigma_1=\sigma_2$. 
\end{proposition}

\begin{proof}
First of all, notice that the quotient $Z:=C/\sigma$ is an integral curve with arithmetic genus $0$ over $k$, therefore $Z\simeq \PP^1$. Consider now the following morphism:
$$ 
\begin{tikzcd}
\phi: C \arrow[r, "f"] & \PP^1 \arrow[r, "i_{g-1}", hook] & \PP^{g-1}
\end{tikzcd}
$$ 
where $f$ is the quotient morphism and $i_{g-1}$ is the $(g-1)$-embedding. It is enough to prove that $\phi^*(\cO_{\PP^{g-1}}(1)) \simeq \omega_C$ as it implies that every hyperelliptic involution comes from the canonical morphism (therefore it is unique).

We denote by $\cL$ the line bundle $\phi^*\cO_{\PP^{g-1}}(1)$. Using Riemann-Roch for integral curves, we get that 
$$ h^0(C,\omega_C \otimes \cL^{-1})= h^0(C,\cL) + 1 -g $$ 
therefore if we prove that $h^0(C,\cL)= g$, we get that $\deg (\omega_C \otimes \cL^{-1})=0$ and that $h^0(\omega_C \otimes \cL^{-1})=1$ which implies $\cL \simeq \omega_C$ (as $C$ is integral). Because $f$ is finite, we have that $$\H^0(C,\cL)=\H^0(\PP^1,i_{g-1}^*\cO_{\PP^{g-1}}(1) \otimes f_*\cO_C) $$
thus using that $f_*\cO_C= \cO_{\PP^1}\oplus \cO_{\PP^1}(-g-1)$ as $f$ is a cyclic cover of degree $2$, we get that 
$$ \H^0(C,\cL)= \H^0(\PP^1, \cO_{\PP^1}(g-1))\oplus \H^0(\PP^1, \cO_{\PP^1}(-2))$$
which implies $h^0(C,\cL)=g$. 
\end{proof}

Now we deal with the genus $1$ case.
\begin{lemma}\label{lem:genus1}
	Let $(C,p_1,p_2)$ be a $2$-pointed $A_r$-stable curve of genus $1$. Then there exists a unique hyperelliptic involution which sends one point into the other. 
\end{lemma}

\begin{proof}
	The proof of this lemma consists in describing all the possible cases. First of all, the condition on the genus implies that the curve is $A_3$-stable, as the arithmetic genus would be too high with more complicated $A_r$-singularities. Clearly, 
	$$\sum_{\Gamma \subset C} g(\Gamma) \leq g(C)=1$$ 
	where $\Gamma$ varies in the set of irreducible components of $C$. Consider first the case where there exists $\Gamma$ such that $g(\Gamma)=1$, therefore all the other irreducible components have genus $0$ and all the separating points are nodes. Thanks to the stability condition, it is clear that $C$ is either integral (case (a) in \Cref{fig:Genus1}) with two smooth points or it has two irreducible components $\Gamma_1$ and $\Gamma_0$, where $g(\Gamma_1)=1$, $g(\Gamma_0)=0$, they intersect in a separating node and the two smooth points lie on $\Gamma_0$ (case (b) in \Cref{fig:Genus1}).
	Suppose then that for every irreducible component $\Gamma$ of $C$ we have $g(\Gamma)=0$. Therefore, if the curve is not $A_1$-prestable, we have that the only possibility is that there exists a separating tacnode between two genus $0$ curves. Stability condition implies that $C$ can be described as two integral curves of genus $0$ intersecting in a tacnode, and the two points lie in different components (case (d) in \Cref{fig:Genus1}). Finally, if $C$ is $1$-prestable we get that $C$ has two irreducible components of genus $0$ intersecting in two points and the smooth sections lie in different components (case (c) in \Cref{fig:Genus1}). 
	\begin{figure}[H]
		\caption{$2$-pointed genus $1$ curves}
		\centering
		\includegraphics[width=1\textwidth]{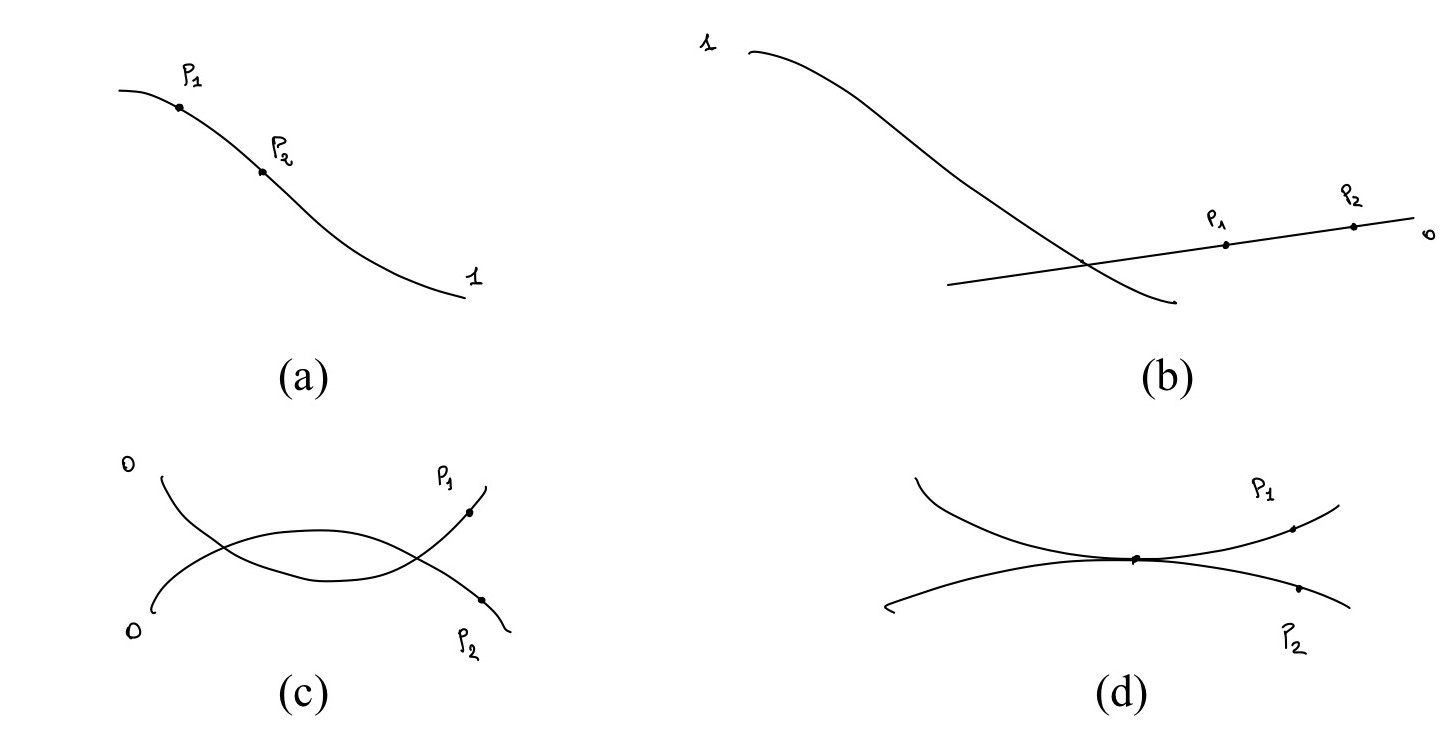}
		\label{fig:Genus1}
	\end{figure}
	In all of these four situations, it is easy to prove existence and uniqueness of the hyperelliptic involution. 
\end{proof}

\begin{remark}
In the previous lemma, we can also consider the case when  $p:=p_1=p_2$ and $(C,p)$ is an $1$-pointed $A_2$-stable curve. The same result is true.
\end{remark}
We want to treat the case of reducible curves.

\begin{definition}\label{def:subcurve}
	 Let $(C,\sigma)$ a hyperelliptic $A_r$-stable curve over an algebraically closed field. A one-equidimensional reduced closed subscheme $\Gamma \subset C$ is called a \emph{subcurve} of $C$. We denote by $\iota_{\Gamma}:\Gamma \rightarrow C$ the closed immersion. If $\Gamma$ is a subcurve of $C$, we denote by $C-\Gamma$ the complementary subcurve of $\Gamma$, i.e. the closure of $C\setminus \Gamma$ in $C$. 
\end{definition}

\begin{remark}
	A subcurve is just a union of irreducible components of $C$ with the reduced scheme structure. Furthermore, given an one-equidimensional closed subset of $C$, we can always consider the associated subcurve with the reduced scheme structure.
\end{remark}

\begin{lemma}\label{lem:subcurve}
	 Let $(C,\sigma)$ be a hyperelliptic $A_r$-stable curve of genus $g\geq 2$ over an algebraically closed field and $\Gamma \subset C$ be a subcurve such that $g(\Gamma)\geq 1$. Then $\dim(\Gamma \cap \sigma(\Gamma))=1$.  
\end{lemma}

\begin{proof}
Suppose $\dim(\Gamma \cap \sigma(\Gamma))=0$ and consider the quotient morphism $f: C \arr Z:=C/\sigma$. Consider now the schematic image $f(\Gamma) \subset Z$, which is a subcurve of $Z$ as it is reduced and $f$ is finite. If we restrict $f$ to $f(\Gamma)$, we get the following commutative diagram:
$$
\begin{tikzcd}
\Gamma \arrow[r, hook] \arrow[rd] & f^{-1}f(\Gamma) \arrow[r, hook] \arrow[d, "f_{\Gamma}"] & C \arrow[d, "f"] \\
& f(\Gamma) \arrow[r, hook]                               & Z               
\end{tikzcd}
$$
where the square diagram is cartesian. If we consider $U:=f(\Gamma) \setminus f(\Gamma \cap \sigma(\Gamma))$, we get that $f^{-1}(U)$ is a disjoint union of two open subsets and $\sigma$ maps one into the other,  therefore the action on $f^{-1}(U)$ is free, giving that the degree of $f_{\Gamma}$ restricted to $f^{-1}(U)$ is $2$. The condition $\dim(\Gamma \cap \sigma(\Gamma))=0$ assures us that $f^{-1}(U)=(\Gamma \cup \sigma(\Gamma))\setminus (\Gamma \cap \sigma(\Gamma))$ is in fact dense in $f^{-1}f(\Gamma)$. Thus the morphism $f\vert_{\Gamma}:\Gamma \arr \pi(\Gamma)$ is a finite morphism which is in fact an isomorphism over $\Gamma\setminus \sigma(\Gamma)$, which is a dense open, therefore it is birational. This implies the following inequality
$$g(Z)\geq g(f(\Gamma))\geq g(\Gamma)\geq 1$$
which is absurd because $g(Z)=0$. 
\end{proof}

\begin{remark}
	Notice that this lemma implies that the only irreducible components that are not fixed by the hyperelliptic involution have (arithmetic) genus equal to $0$. Therefore, if we have a hyperelliptic involution $\sigma$ of $C$ and $\Gamma$ is an irreducible component of positive genus, we get that $\sigma(\Gamma)=\Gamma$ and $\sigma\vert_{\Gamma}$ is a hyperelliptic involution of $\Gamma$. This is true because we are quotienting by a linearly reductive group, thus the  morphism 
	$$\Gamma/\sigma \rightarrow C/\sigma$$
	induced by the closed immersion $\Gamma \subset C$ is still a closed immersion.
\end{remark}	

Let $C$ be an $A_r$-stable curve of genus $g\geq 2$. We denote by ${\rm Irr}(C)$ the set whose elements are the irreducible components of $C$. Then every automorphism $\phi$ of $C$ induces a permutation $\tau_{\phi}$ of the set ${\rm Irr}(C)$. First of all, we need to prove that the action on ${\rm Irr}(C)$ is the same for every hyperelliptic involution. Then we prove that the action on the $0$-dimensional locus described by all the intersections between the irreducible components is the same for all hyperelliptic involutions. Finally we see how these two facts imply the uniqueness of the hyperelliptic involution. We denote by $l(Q)$ the length of a $0$-dimensional subscheme $Q\subset C$.

\begin{lemma}\label{lem:exist-decomposition}
	Let $(C,\sigma)$ be a hyperelliptic $A_r$-stable curve of genus $g$ over $k$ and suppose there exists two irreducible components $\Gamma_1$ and $\Gamma_2$  of genus $0$ of $C$ such that $\sigma$ send $\Gamma_1$ in $\Gamma_2$. Let $n$ be the length of $\Gamma_1\cap \Gamma_2$. Then there exists a nonnegative integer $m$ and $m$ disjoint subcurves $D_i \subset C$ such that the following properties holds:
	\begin{itemize}
		\item[a)] $m+n\geq 3$,
		\item[b)] $\sigma(D_i)=D_i$ for every $i=1,\dots,m$,
		\item[c)] the length of the subscheme $D_i\cap \Gamma_j$ is $1$ for every $i=1,\dots,m$ and $j=1,2$ and $D_i \cap \Gamma_1 \cap \Gamma_2=\emptyset$,
		\item[d)] if we denote by $P_j^i$ the intersection $D_i \cap \Gamma_j$, we have that $(D_i,P_1^i,P_2^i)$ is a $2$-pointed $A_r$-stable curve of genus $g_i>0$ such that $\sigma\vert_{D_i}$ is a hyperelliptic involution of $D_i$ which maps $P_1^i$ to $P_2^i$,
		\item[e)] $C=\Gamma_1\cup \Gamma_2 \cup \bigcup_{i=1}^m D_i$.
	\end{itemize}
Furthermore, the following equality holds $$g=m+n-1+\sum_{i=0}^m g_i. $$
\end{lemma}

Before proving the lemma, let us explain it more concretely. First of all, the intersections in the lemma are scheme-theoretic, therefore we can have non-reduced ones. More precisely, the intersections have to be supported in singular points of type $A_{2h-1}$, as the local ring is not integral. It is easy to prove that the scheme-theoretic intersection in such a point is in fact of length $h$. Notice that both $m$ or $n$ can be zero.

\begin{figure}[H]
	\caption{Decomposition as in \Cref{lem:exist-decomposition}}
	\centering
	\includegraphics[width=1\textwidth]{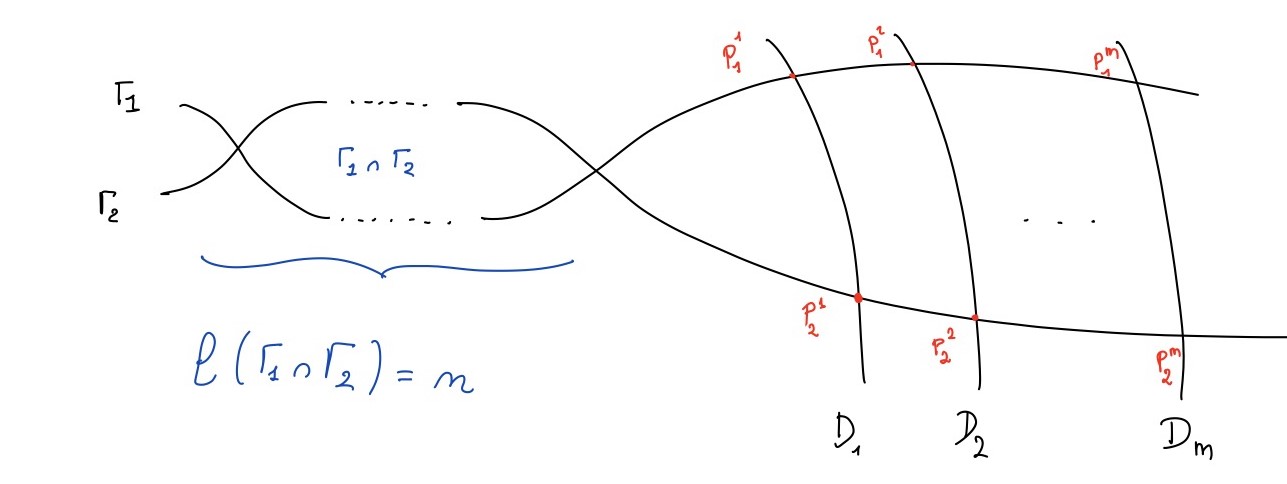}
	\label{fig:Decomp}
\end{figure}
	
\begin{proof}
	  Let $f:C\rightarrow Z$ be the quotient morphism and $\Gamma$ the image of $\Gamma_1\cup\Gamma_2$ through $\pi$. Because every node in $Z$ is a separating node, we know that $Z-\Gamma = \bigsqcup_{i=1}^m E_i$, i.e. it is a disjoint union (possibly empty) of $m$ subcurves of $Z$, which are still reduced, connected and of genus $0$.  Let $D_i$ be the subcurve of $C$ associated to the closed subset $\pi^{-1}(E_i)$ of $C$. We prove that $D_1,\dots,D_m$ verify the properties listed in the proposition.
	
	 Clearly, $D_i\cap D_j=\emptyset$ for every $i\neq j$ by construction. Properties b) and e) are verified by construction as well. Notice that $\Gamma_1\cap\Gamma_2\cap D_i=\emptyset$ for every $i=1,\dots,m$, otherwise if $p\in \Gamma_1\cap\Gamma_2\cap D_i$ is a closed point, then the local ring $\cO_{C,p}$ would have $3$ minimal primes. This cannot occur as the only singularities allowed are of type $A_r$, which have at most $2$ minimal primes. Therefore if we define $Q_i:=E_i\cap \Gamma$, we have that $\pi^{-1}(Q_i)$ is disconnected (as it does not belong to $\pi(\Gamma_1\cap\Gamma_2)$) and thus we are in the situation (n1) of \Cref{prop:description-quotient}. Property c) follows easily from this. Because $C$ is $A_r$-stable, we also get property a). 
	 
	 Regarding property d), we know that $D_i$ is reduced and that $D_i\cap (C-D_i)$ has length $2$, hence it is enough to prove $D_i$ is connected and then the statement follows. However, suppose $D_i$ is not connected, namely  $D_i=D_i^1 \bigsqcup D_i^2$ with $D_i^j$ two subcurves of $C$ for $j=1,2$ such that $\sigma(D_i^1)=D_i^2$. Using \Cref{lem:subcurve} again, we get $g(D_i^j)=0$ for $j=1,2$. This is not possible because of the stability condition on $C$. The genus formula follows from a straightforward computation using \Cref{rem:genus-count}.
\end{proof}	

Now that we have described the geometric structure of $C$, we can use it to prove that any other hyperelliptic involution has to act in the same way over the set ${\rm Irr}(C)$ of the irreducible components. 
\begin{lemma}
	Let $C/k$ be a $A_r$-stable curve of genus $g$ and suppose we have a decomposition as in \Cref{lem:exist-decomposition}. Thus every hyperelliptic involution $\sigma$ of $C$ commutes with the decomposition, i.e. we have $\sigma(\Gamma_1)=\Gamma_2$.
\end{lemma}

\begin{proof}
	Let $\pi:C\rightarrow Z$ be the quotient morphism. We start with the case $m=0$ and $n:=l(\Gamma_1\cap \Gamma_2)\geq 3$. Suppose $\sigma(\Gamma_j)=\Gamma_j$ for $j=1,2$. Recall that the only singularities that can appear in the intersection $\Gamma_1\cap \Gamma_2$ are of the form $A_{2h-1}$. As the involution does not exchange the two irreducible components, \Cref{prop:descr-inv} implies that the only possible singularities in the intersection are nodes or tacnodes with nodes as quotients. Because $n\geq 3$ we have that the quotient has two irreducible components intersecting in at least two nodes, but this does not have genus $0$. Therefore $\sigma(\Gamma_1)=\Gamma_2$.
	
	Suppose now $m\geq 1$ and $n\geq 2$. Because $n\geq 2$ then $\Gamma_1 \cup \Gamma_2$ has positive genus therefore $\sigma(\Gamma_1\cup\Gamma_2)$ and $\Gamma_1\cup\Gamma_2$ have a common component because of \Cref{lem:subcurve}. Suppose $\sigma(\Gamma_1)\neq \Gamma_2$ and thus without loss of generality $\sigma(\Gamma_2)=\Gamma_2$. If $\sigma(\Gamma_1)=\Gamma_1$, then $\pi(\Gamma_1)\cap\pi(\Gamma_2)$ contains at least a node because $n\geq 2$ and the subcurve $\pi(\Gamma_1\cup\Gamma_2\cup D_1)$ of $Z$ does not have genus $0$ as $D_1$ is connected. Therefore we can suppose $\sigma(\Gamma_1)\subset D_1$. Then $\sigma(\Gamma_1\cap\Gamma_2)\subset D_1\cap\Gamma_2$, but this implies $l(\Gamma_1\cap\Gamma_2)\leq 1$ which is in contradiction with $n\geq 2$.
	
	Now we consider the case $n=1$ and $m\geq 2$. Again it is easy to prove that we cannot have $\sigma(\Gamma_i)=\Gamma_i$ for $i=1,2$. Without loss of generality we can suppose $\sigma(\Gamma_1)\subset D_1$. However $\sigma(D_2)\neq D_2$ as $D_2\cap \Gamma_1\neq \emptyset$ and thus $\sigma(D_2)\cap D_2$ has to share an irreducible component because of \Cref{lem:subcurve}. This implies that $\Gamma_2 \subset \sigma(D_2)$ because $D_2$ is connected. Thus $\sigma(\Gamma_1)\cap\sigma(\Gamma_2)\subset D_1\cap D_2 = \emptyset$, which is absurd. The only possibility is $\sigma(\Gamma_1)=\Gamma_2$.
	
	Finally, we consider the case $n=0$ and $m\geq 3$. As above, we cannot have $\sigma(\Gamma_j)=\Gamma_j$ for $j=1,2$, otherwise the quotient does not have genus $0$. If $\sigma(\Gamma_1)\subset D_1$, then $\sigma(\Gamma_2)$ is contained in only one of the subcurves $\{D_i\}_{i=1\dots,m}$. Thus there exists at least one between the $D_i$'s, say $D_2$, which is stable under the action of the involution (because $m\geq 3$, $D_i$ is connected for every $i=1,\dots,m$ and $\dim D_i \cap \sigma(D_i) = 1$). This is absurd because $\sigma(\Gamma_1)\cap \sigma(D_2) \subset D_1\cap D_2=\emptyset$. This implies $\sigma(\Gamma_1)=\Gamma_2$ and we are done.
\end{proof}

\begin{corollary}\label{cor:action-irrcomp}
	 If $C/k$ is an $A_r$-stable curve of genus $g\geq 2$ and $\sigma_1$ and $\sigma_2$ are two hyperelliptic involution, then $\tau_{\sigma_1}=\tau_{\sigma_2}$.
\end{corollary}

Let us study now the action on the intersections between irreducible components.

\begin{lemma}\label{lem:action-intersection}
	Let $\Gamma_1$ and $\Gamma_2$ two irreducible components of $C$ and $p\in \Gamma_1\cap \Gamma_2$ be a closed point. If $\sigma_1$ and $\sigma_2$ are two hyperelliptic involution of $C$, then $\sigma_1(p)=\sigma_2(p)$. 
\end{lemma}
\begin{proof}
 Suppose $\sigma$ is a hyperelliptic involution such that $\sigma(\Gamma_1)=\Gamma_2$. Because the quotient of $\Gamma_1\cup\Gamma_2$ by $\sigma$ is irreducible of genus $0$, we have that $\sigma(p)=p$ for every $p \in \Gamma_1\cap\Gamma_2$. Suppose now $\sigma(\Gamma_1)\neq \Gamma_2$ and denote by $\pi:C\rightarrow Z$ the quotient morphism. Clearly $\pi(\Gamma_1\cup\Gamma_2)$ is not irreducible and $\pi(\Gamma_1\cap\Gamma_2)$ is a separating node, therefore $\Gamma_1\cap\Gamma_2$ is either supported on a node, a tacnode or two disjoint nodes exchanged by the involution. In the first two cases, the intersection is supported on one point therefore there is nothing to prove. If the intersection is supported on two nodes, then every involution has to exchange them because otherwise the quotient would not have genus $0$.
\end{proof}

\begin{remark}\label{rem:part-case}
	We can say more about the case $\sigma(\Gamma_1)=\Gamma_2$. In this situation, we claim every hyperelliptic involution acts trivially on $\Gamma_1\cap \Gamma_2$ scheme-theoretically, not only set-theoretically. In fact, suppose we have a fix point $p\in \Gamma_1\cap\Gamma_2$.  We know that $\Gamma_1$ and $\Gamma_2$ are irreducible components of genus $0$ and the image of $p$ in the quotient by $\sigma$ is a smooth point. Therefore we can consider the local ring $A:=\cO_{C,p}$ and we know that the invariant subalgebra is a DVR which is denoted by $R$. By flatness, we know that $A=R[y]/(y^2-h)$ where $h$ is an element of $R$ and $\sigma$ is defined by the association $y\mapsto -y$. By hypothesis, $A$ has two minimal primes, namely $p$ and $q$, such that $p\cap q=0$. We can consider the morphism 
	$$A \rightarrow A/p\oplus A/q$$
	which is injective after completion, therefore injective. Because $R$ is a DVR, an easy computation shows that $h$ is a square, thus $A$ is of the form $R[y]/(y^2-r^2)$ with $r\in R$ and the involution is defined by the formula $\sigma(a+by)=a-by$. Clearly we have that the two minimal primes are $y-r$ and $y+r$ and the intersection $\Gamma_1\cap\Gamma_2$ is contained in the fixed locus of the action because is defined by the ideal $(y,r)$ (the fixed locus is defined by the ideal $(y)$).
 \end{remark}
Finally we can prove the theorem.
\begin{theorem}
	Let $C/k$ be a $A_r$-stable curves of genus $g\geq 2$. If $\sigma_1$ and $\sigma_2$ are two hyperelliptic involution of $C$, then $\sigma_1=\sigma_2$.
\end{theorem}

\begin{proof}
	 We prove that two hyperelliptic involutions coincide restricted to a decomposition of $C$ in subcurves.
	The integral case is done in \Cref{prop:integral}. Because of \Cref{cor:action-irrcomp} and \Cref{lem:action-intersection}, we know that every hyperelliptic involution $\sigma$ acts in the same way on the set of irreducible components and on the (set-theoretic) intersection on every pair of irreducible components. 
	
	 If we consider an irreducible component of genus greater than $2$, then $\sigma$ restricts to the irreducible component, it still is a hyperelliptic involution  and we can use \Cref{prop:integral} to get the uniqueness restricted to the component. If we restrict $\sigma$ to a component $\Gamma$ of genus $1$, we still have that it is a hyperelliptic involution. If we consider a point $p \in \Gamma\cap C-\Gamma$, we get that every other hyperelliptic involution $\sigma'$ has to verify $\sigma'(p)=\sigma(p)$ thanks to \Cref{lem:action-intersection}. Then we can apply \Cref{lem:genus1} to get the uniqueness restricted to the component $\Gamma$. 
	  
	  Finally, if $\Gamma$ has genus $0$, one can have only two possibilities, namely $\sigma(\Gamma)=\Gamma$ or $\sigma(\Gamma)\neq \Gamma$. If $\sigma(\Gamma)=\Gamma$, then two different involutions have to coincide when restricted to $\Gamma$ in at least two points by stability (we have tacnodes or nodes as intersections). This easily implies they coincide on the whole component. If $\sigma(\Gamma)\neq \Gamma$, it is easy to see that every hyperelliptic involution restricted to $\sigma(\Gamma) \cup \Gamma$ are determined by an involution of $\PP^1$ (one can just choose two identification $\Gamma\simeq \PP^1$ and $\sigma(\Gamma)\simeq \PP^1$). Because $\sigma$ does not fix $\Gamma$, we can use \Cref{lem:exist-decomposition} to get a decomposition $C$ in subcurves $D_1,\dots,D_m$ with the properties listed in the lemma. The stability condition, i.e. $m+l(\Gamma\cap \sigma(\Gamma))\geq 3$, implies that two hyperelliptic involution restricted to $\Gamma\cup \sigma(\Gamma)$ are determined by two involution of $\PP^1$ which coincide restricted to a subscheme of length $m+l(\Gamma\cap\sigma(\Gamma))$, therefore they coincide.
\end{proof}

\begin{remark}
	 Notice that for the case $m=0$ and $\Gamma\cap\sigma(\Gamma)$ supported on a single point we really need \Cref{rem:part-case}. In fact for the case of two $\PP^1$ glued on a subscheme of length greater or equal than $3$ concentrated on a single point, the condition that two hyperelliptic involution have to coincide on the set-theoretic intersection is not enough to conclude they are the same. 
\end{remark}
\subsection*{Unramifiedness}

Next we focus on the unramifiedness of the map $\eta$. We are going to prove that the involution is unique´ for deformations using deformation theory of curves and of morphisms of curves.

\begin{remark}\label{rem:deformation}
Let $\cX,\cY$ be two algebraic stacks and let $f:\cY\arr \cX$ by a morphism. Suppose it is given a 2-commutative diagram
$$
\begin{tikzcd}
\spec k\arrow[r, "y"] \arrow[d] & \cY \arrow[d] \\
{\spec k[\epsilon]} \arrow[r, "x_{\epsilon}"]          & \cX                       
\end{tikzcd}
$$ 
with $k$ a field and $k[\epsilon]$ the ring of dual numbers over $k$. We define $x:=f(y)$ and $\cY_x$ the fiber product of $f$ with the morphism induced by $x$; clearly we have a lifting of $y$ from $\cY$ to $\cY_x$ which is denoted by $y$ by abuse of notation. By standard argument in deformation theory, we get the following exact sequence of vector spaces over $k$:
$$
\begin{tikzcd}
0 \arrow[r] & T_{\id}\aut_{\cY_x}(y) \arrow[r] & T_{\id}\aut_{\cY}(y) \arrow[r] & T_{\id}\aut_{\cX}(x) \arrow[lld, "\alpha_f(y)" description ] \\
& \pi_0(T_{y}\cY_x) \arrow[r]        & \pi_0(T_y\cY) \arrow[r]          & \pi_0(T_x\cX)                     
\end{tikzcd}
$$ 
where $T_x \cX$ is the groupoid of morphisms $x_{\epsilon}:\spec k[\epsilon] \arr \cX$ such that the composition with $\spec k \hookrightarrow \spec k[\epsilon]$ is exactly $x:\spec k \arr \cX$. By standard notation, we call $T_x\cX$ the tangent space of $\cX$ at $x$.
\end{remark} 

If we prove that $\pi_0(T_y\cY_x)=0$, then the morphism $f$ is fully faithful at the level of tangent spaces, and therefore unramified. We prove that this is true for the morphism $\eta$. 

First of all, we need to describe the fiber $\Htilde_C$ of the morphism $\eta: \Htilde_g^r \rightarrow \Mtilde_g^r$ in a point $C \in \Mtilde_g^r(k)$, where $k/\kappa$ is an extension of fields with $k$ algebraically closed. As the map $\eta$ is faithful, we know that $\Htilde_C$ is equivalent to a set. Given a point $(C,\sigma)\in \Htilde_C(k)$, we have that an element in $T_{(C,\sigma)}\Htilde_C$ is a pair $(C[\epsilon],\sigma_{\epsilon})$ where $C[\epsilon]$ is the trivial deformation of $C$ and $\sigma_{\epsilon}:C[\epsilon]\rightarrow C[\epsilon]$ is a deformation of $\sigma$. Therefore, we need to prove the uniqueness of the hyperelliptic involution for deformations. To do so, we study the deformations of the quotient map $\pi: C \rightarrow Z:=C/\sigma$. 

Let ${\rm Def}^{\rm fix}_{C/Z}$ be the deformation functor associated to the problem of deforming the morphism $\pi:C \rightarrow Z$ with both source and target fixed and ${\rm InfAut}(Z)$ be the deformation functor of infinitesimal automorphisms of $Z$. There is a natural morphism of deformation functors
$$ \alpha:{\rm InfAut}(Z) \longrightarrow {\rm Def}^{\rm fix}_{C/Z}$$
whose restriction to the tangent spaces $d\alpha$ induces a morphism of $k$-vector spaces. Furthermore, we have a map $$\gamma:T_{(C,\sigma)}\Htilde_C \longrightarrow {\rm Def}^{\rm fix}_{C/Z} $$ defined by the association $(C[\epsilon],\sigma_{\epsilon})\mapsto \pi_{\epsilon}:C[\epsilon]\rightarrow Z[\epsilon]\simeq C[\epsilon]/\sigma_{\epsilon}$.

\begin{remark}
It is not completely trivial that the quotient $C[\epsilon]/\sigma_{\epsilon}$ is isomorphic to the trivial deformation of $Z$. One can prove it using the fact that the morphism $\pi$ is finite reducing to the affine case.
\end{remark}

Notice that $\gamma(\sigma_{\epsilon}) \in \im{d\alpha}$ implies $\sigma_{\epsilon}=0$ because of \Cref{lem:unique-inv-quotient}. Therefore, it is enough to prove $\im{\gamma}\subset \im{d\alpha}$ to get that $T_{(C,\sigma)}\Htilde_C$ is trivial.

Let us focus on the morphism $\alpha$. The morphism $d\alpha$ can be identified with the map 
$$ \hom_{\cO_Z}(\Omega_Z,\cO_Z) \longrightarrow \hom_{\cO_Z}(\Omega_Z,\pi_*\cO_C)$$ 
induced by applying $\hom_{\cO_Z}(\Omega_Z,-)$ to the natural exact sequence
$$ 0 \rightarrow \cO_Z \rightarrow \pi_*\cO_C \rightarrow L\rightarrow 0 .$$
Clearly, if we have $\hom_{\cO_Z}(\Omega_Z, L)=0$, we get that $d\alpha$ is surjective and we have done. This is not true in general and we will see why in \Cref{ex:not-involution}. Therefore we need to treat the problem with care, studying the deformation spaces involved. 

Let us start with the case of $(C,\sigma)$, where $C$ does not have separating node. Thanks to the description in \Cref{prop:description-quotient}, we have that the quotient map $\pi:C \rightarrow Z$ is finite flat of degree $2$. The theory of cyclic covers implies that every deformation of $\pi$ still induces a hyperelliptic involution, meaning that in this case the composition 
$$ \bar{\gamma}: T_{(C,\sigma)}\Htilde_C \longrightarrow \frac{{\rm Def}^{\rm fix}_{C/Z}}{\im{d\alpha}} $$
is an isomorphism of vector spaces. 

However, if $C$ has at least one separating node, the deformations of $\pi$ do not always give a deformation of the involution of $\sigma$. 

\begin{example}\label{ex:not-involution}
	Let $C$ be an ($A_1$)-stable curve of genus $2$ over $k$ with a separating node $p \in C(k)$. We have a hyperelliptic involution $\sigma$ which fixes the two genus $1$ components such that the two points $p_1,p_2$ over the node $p$ in the normalization of $C$ are fixed by the involution as well. Therefore we have a quotient morphism $\pi:C \rightarrow Z$, where $Z$ is a genus $0$ curve with two components meeting in a separating node $q$, image of $p$ through $\pi$. The morphism $\pi$ is finite and it is flat of degree $2$ restricted to the open $C \setminus p$. On the contrary, locally around the point $p$ is finite completely ramified of degree $3$, i.e. $\pi^{-1}(q)$ is the spectrum of a length $3$ local artinian $k$-algebra. Clearly, we can deform $\pi$ without deforming source and target but making it unramified over $q$. This deformation cannot correspond to a hyperelliptic involution. 
	
	In another way, we are deforming the two involutions on the two components of $C$ such that the fixed locus moves away from $p_1$ and $p_2$. This implies that we cannot patch them together to get an involution of $C$. To sum up, we need to consider only deformations where the fixed locus of the two involutions does not move. 
	
	Now we look at the deformation spaces. Let $C_1$ and $C_2$ be the two genus $1$ component of $C$ and $Z_1$ and $Z_2$ are the two components of $Z$ (reduced schematic images of $C_1$ and $C_2$ respectively). Let $L_i$ be the quotient line bundle of $\cO_{Z_i} \hookrightarrow \pi_*\cO_{C_i}$ for $i=1,2$. An easy computation (see \Cref{lem:decomp-line-bundle}) shows that $$\hom_{\cO_Z}(\Omega_Z,L)=\hom_{\cO_{Z_1}}(\Omega_{Z_1},L_1)\oplus \hom_{\cO_{Z_2}}(\Omega_{Z_2},L_2)$$
	and in particular $\hom_{\cO_Z}(\Omega_Z,L)$ is not zero. Nevertheless, if we ask that our deformation of $\pi$ is totally ramified over the separating node, we get exactly the space $$\hom_{\cO_{Z_1}}(\Omega_{Z_1},L_1(-p_1))\oplus \hom_{\cO_{Z_2}}(\Omega_{Z_2},L_2(-p_2))$$ which is zero giving us the unicity of the hyperelliptic involution for deformations. We are going to generalize this computation to our setting.
\end{example}

Suppose $\Gamma$ is an  subcurve of $Z$, then we define the subcurve $C_{\Gamma}:=\pi^{-1}(\Gamma)_{\rm red}$ of $C$. Let $\pi_{\Gamma}:C_{\Gamma}\rightarrow \Gamma$ be the restriction of $\pi$ to $C_{\Gamma}$ and $L_{\Gamma}$ be the quotient bundle of the natural map $\cO_{\Gamma}\hookrightarrow \pi_{\Gamma,*}\cO_{C_{\Gamma}}$.
The main statement in this section is the following theorem.

\begin{theorem}
	In the situation above, the morphism $\bar{\gamma}$ factors through the following inclusion of vector spaces
	$$ \bigoplus_{\Gamma \in {\rm Irr}(Z)}\hom_{\cO_{\Gamma}}(\Omega_{\Gamma},L_{\Gamma}(-D_{\Gamma}))\subset \hom_{\cO_Z}(\Omega_{Z},L)$$
where $D_{\Gamma}$ is the Cartier divisor on $\Gamma$ defined as $\sum^{\Gamma'\neq \Gamma}_{\Gamma' \in {\rm Irr(Z)}}\Gamma \cap \Gamma'$, or equivalently $\Gamma \cap (Z - \Gamma)$.
\end{theorem}
The theorem above implies the result we need to conclude the study of the unramifiedness of $\eta$.
\begin{corollary}
	For every $(C,\sigma) \in \Htilde_C$, the morphism $\bar{\gamma}\equiv 0$, which implies that $\im \gamma \subset \im{d\alpha}$. 
\end{corollary}

\begin{proof}
	It is enough to prove that for every $\Gamma$ irreducible component of $Z$, we have  that $\hom_{\cO_{\Gamma}}(\Omega_{\Gamma},L_{\Gamma}(-D_{\Gamma}))=0$. Let us explain why this follows from the stability condition on $C$. Clearly $C_{\Gamma}$ is a $A_r$-prestable hyperelliptic curve, with a $2:1$-morphism over $\Gamma\simeq \PP^1$. Let $n_{\Gamma}$ be the number of nodal point on the component $\Gamma$, or equivalently the degree of $D_{\Gamma}$. Thus $\deg(\Omega_{\PP^1}^{\vee} \otimes L_{\Gamma}(-D_{\Gamma}))=+2+(-h_{\Gamma}-1-n_{\Gamma})=1-h_{\Gamma}-n_{\Gamma}$, where $h_{\Gamma}$ is the  (arithmetic) genus of $C_{\Gamma}$. The stability condition on $C$ implies that $2h_{\Gamma}-2+2n_{\Gamma}>0$ because the restriction to $C_{\Gamma}$ of the fiber of a node has at most length $2$. Therefore $\deg (\Omega_{\PP^1}^{\vee} \otimes L_{\Gamma}(-D_{\Gamma}))<0$ and we are done. Equivalently, it is clear that if $h_{\Gamma}>1$, then the degree we want to compute is negative. If $h_{\Gamma}=1$, then $C_{\Gamma}$ has at least a point of intersection with $C-C_{\Gamma}$, therefore $n_{\Gamma}>0$. If $h_{\Gamma}=0$, then $C_{\Gamma}$ intersect $C-C_{\Gamma}$ in subscheme of length at least $3$, but because the restriction to $C_{\Gamma}$ of the fiber of a node of $Z$ has at most length $2$, the support of the intersection contains at least two points, which implies $n_{\Gamma}>1$. 
\end{proof}
To prove the theorem, we reduce to the case when the map $\pi$ is flat, or equivalently there are no separating node on $C$ and then we prove the statement in that case.

\begin{definition}\label{def:sep-dec}
	Let $C$ be an $A_r$-prestable curve and let $\{C_i\}_{i \in I}$ be a set of subcurves of $C$ (see \Cref{def:subcurve} for the definition of subcurves). We  say that $\{C_i\}_{i \in I}$ is an $A_1$-separating decomposition of $C$ if the following three properties are satisfied:
\begin{enumerate}
	\item $C=\bigcup_{i \in I} C_i$,
	\item $C_i$ does not have separating nodes (for $C_i$ itself),
	\item $C_i\cap C_j$ is either empty or a separating node (of $C$) for every $i\neq j$. 
\end{enumerate} 
\end{definition}
 Such a decomposition exists and it is unique. 
 Let $(C,\sigma)$ be a hyperelliptic $A_r$-stable curve and $\{C_i\}_{i \in I}$ be its $A_1$-separating decomposition. As usual, we denote by $\pi:C \rightarrow Z:=C/\sigma$ the quotient morphism. Fix an index $i \in I$. Let $Z_i$ be the schematic image of $C_i$ through $\pi$ and $\pi_i:C_i \rightarrow Z_i$ be the restriction of $\pi$ to $C_i$. Using again the local description in \Cref{prop:description-quotient}, we get that $Z_i$ is reduced (therefore a subcurve of $Z$) and $\pi_i$ is flat. Finally, let $L_i$ the quotient of the natural map 
$$\cO_{Z_i} \hookrightarrow \pi_{i,*}\cO_{C_i}$$
which is a line bundle because $\pi_i$ is flat finite of degree $2$ (and we are in characteristic different from $2$).

The following lemma generalizes the idea in \Cref{ex:not-involution}.

\begin{lemma}\label{lem:decomp-line-bundle}
In the situation above, we have an isomorphism of coherent sheaves over $Z$
$$ L \simeq \bigoplus_{i \in I} \iota_{Z_i,*}L_i$$

where $\iota_{Z_i}:Z_i \hookrightarrow Z$ is the closed immersion of the subcurve $Z_i$ in $Z$.

\end{lemma}

\begin{proof}
 
The commutative diagram
$$\begin{tikzcd}
\bigsqcup_{i \in I} C_i \arrow[r, "\bigsqcup \iota_{C_i}"] \arrow[d, "\bigsqcup \pi_i"'] & C \arrow[d, "\pi"] \\
\bigsqcup_{i\in I}Z_i \arrow[r, "\bigsqcup \iota_{Z_i}"]                    & Z                 
\end{tikzcd}
$$ 
 induces a commutative diagram at the level of structural sheaves
$$
\begin{tikzcd}
\cO_Z \arrow[r, hook] \arrow[d, hook] & \bigoplus_{i \in I} \iota_{Z_i,*}\cO_{Z_i} \arrow[d, hook] \\
\pi_*\cO_{C} \arrow[r, hook]          & \bigoplus_{i \in I}\pi_*\iota_{C_i,*}\cO_{C_i}.                 
\end{tikzcd}
$$ 
We want to prove that the map induced between the quotients of the two vertical maps is an isomorphism, in fact the quotient of the right-hand vertical map is trivially isomorphic to $\bigoplus_{i \in I} \iota_{Z_i,*}L_i$ by construction. By the Snake lemma, it is enough to prove that the induced morphism between the two horizontal quotients is an isomorphism. This follows from a local computation in the separating nodes.   
\end{proof}

\begin{remark}
	Because of the fundamental exact sequence for the differentials associated to the immersion $\iota_{Z_i}:Z_i\hookrightarrow Z$, we have that 
	$$ \hom_{\cO_{Z_i}}(\Omega_{Z_i},L_i)\simeq  \hom_{\cO_{Z_i}}(\iota_{Z_i}^*\Omega_{Z},L_i)$$
	and therefore the previous lemma implies that 
	$$ \hom_{\cO_Z}(\Omega_Z,L) \simeq \bigoplus_{i \in I} \hom_{\cO_{Z_i}}(\Omega_{Z_i},L_{i}).$$
	It is easy to see that the last isomorphism can be described using deformations just restricting a deformation of $\pi$ with both source and target fixed to a deformation of $\pi_i$ with both source and target fixed.
	
\end{remark}

\begin{proposition}
  In the notation above, the map $\bar{\gamma}$ factors through the inclusion of vector spaces
  
  $$ \bigoplus_{i\in I}\hom_{\cO_{Z_i}}(\Omega_{Z_i},L_i(-D_i)) \subset \hom_{\cO_Z}(\Omega_Z, L) $$
  
  where $D_i$ is the Cartier divisor on $Z_i$ defined as $\Sigma_{j \in I, j\neq i} (Z_i \cap Z_j)$.
\end{proposition}
\begin{proof}
	Let us start with a deformation of the hyperelliptic involution $\sigma_{\epsilon}:C[\epsilon]\rightarrow C[\epsilon]$. Let $p:\spec k \hookrightarrow C$ be a separating nodal point, then a local computation around $p$ in $C[\epsilon]$ shows that the following diagram is commutative
	$$\begin{tikzcd}
	{\spec k[\epsilon]} \arrow[r, "{p[\epsilon]}", hook] \arrow[d, "{p[\epsilon]}"', hook] & {C[\epsilon]} \\
	{C[\epsilon]} \arrow[ru, "\sigma_{\epsilon}"]                                          &              
	\end{tikzcd}$$
	where $p[\epsilon]$ is the trivial deformation of $p$. This means that every deformation of the hyperelliptic involution (where the curve $C$ is not deformed) cannot deform around the separating nodes. This also implies that the same is true for $\pi_{\epsilon}$ the quotient morphism (which can be associated to the element $\bar{\gamma}(\sigma_{\epsilon})$), namely $\pi_{\epsilon}\circ p[\epsilon]= q[\epsilon]$ where $q=\pi(p)$. Furthermore, the same is true for the restriction $\pi_{i,\epsilon}$ of $\pi_{\epsilon}$ to $C_i[\epsilon]$ for every $i \in I$. Therefore, the statement follows from the following fact: given the element $\delta_{i,\epsilon}$ representing $\pi_{i,\epsilon}$ in $\hom_{\cO_{Z_i}}(\Omega_{Z_i},L_i)$, the condition $ \pi_{i,\epsilon}\circ p[\epsilon]= q[\epsilon]$ translates into the condition $\delta_{i,\epsilon}(p)=0$, which implies $\delta_{i,\epsilon} \in \hom_{\cO_{Z_i}}(\Omega_{Z_i},L_i(-p))$. 
\end{proof}

\begin{remark}
		Notice that the intersection $Z_i \cap Z_j$ is a smooth point on $Z_i$, therefore a Cartier divisor.
\end{remark}

Finally, we have reduced ourself to study the case of $C$ having no separating nodes, or equivalently the quotient morphism $\pi:C \rightarrow Z$ being flat. We have to describe the whole $\hom_{\cO_Z}(\Omega_{Z},L)$ when $(C,\sigma)$ is a hyperelliptic $A_r$-prestable curve without separating node, because every deformation of the morphism $\pi$ gives rise to a deformation of $\sigma$. 

Let $\{ Z_i \}_{i \in I}$ be the $A_1$-separating decomposition of $Z$ (or equivalently the decomposition in irreducible components as $Z$ has genus 0). We denote by $\pi_i:C_i \rightarrow Z_i$ the restriction of $\pi:C\rightarrow Z$ to $Z_i$, i.e. $C_i:=Z_i\times_Z C$. Again $L_i$ is the quotient line bundle of the natural morphism $\cO_{Z_i} \hookrightarrow \pi_{i,*}\cO_{C_i}$. 

\begin{remark}
 	Notice that in this situation we are just considering the pullback of $\pi$, while in the case of separating nodes we needed to work with the restriction to the subcurve $C_i$. The reason is that in the case of a $A_1$-separating decomposition of $C$, $\pi^{-1}(Z_i)=\pi^{-1}\pi(C_i)$ was set-theoretically equal to $C_i$, but not schematically. In fact, $\pi^{-1}Z_i$ is not reduced, and it has embedded components supported on the restrictions of the separating node. 
\end{remark}

This last proposition finally gives us the theorem.
\begin{proposition}
	In the notation above (if $C$ has no separating nodes),
	$$\hom_{\cO_Z}(\Omega_{Z},L) \simeq \bigoplus_{i \in I}\hom_{\cO_{Z_i}}(\Omega_{Z_i},L_i(-D_i))$$
	where $D_i$ is the Cartier divisor on $Z_i$ defined as $\sum_{j \in I. j \neq i}(Z_i \cap Z_j)$.
\end{proposition}

\begin{proof}
  Firstly, we consider the exact sequence
  $$ 0 \rightarrow \cO_Z \rightarrow \bigoplus_{i \in I}\iota_{Z_i,*} \cO_{Z_i} \rightarrow \bigoplus_{n \in N(Z)}k(n) \rightarrow 0 $$
  where $N(Z)$ is the set of nodal point of $Z$.
  Because $L$ is a line bundle ($\pi$ is flat), if we tensor the exact sequence with $L$ and then apply the functor $\hom_{\cO_Z}(\Omega_{Z},-)$, we end up with an exact sequence
  $$ 0 \rightarrow \hom_{\cO_Z}(\Omega_{Z},L) \rightarrow \bigoplus_{i \in I}\hom_{\cO_{Z_i}}(\iota_{Z_i}^*\Omega_Z,\iota_{Z_i}^*L) \rightarrow \bigoplus_{n \in N(Z)}\hom_{\cO_Z}(\Omega_Z,k(n)). $$
  The flatness of $\pi$ implies that $\iota_{Z_i}^*L \simeq L_{i}$ and using the fundamental exact sequence of the differentials we get
  $$\hom_{\cO_{Z_i}}(\iota_{Z_i}^*\Omega_Z,L_i) \simeq \hom_{\cO_{Z_i}}(\Omega_{Z_i},L_i).$$
  Notice that $\hom_{\cO_Z}(\Omega_{Z},k(n))=k(n)dx\oplus k(n)dy$ where $dx,dy$ are the two generators of $\Omega_Z$ locally at the node $n \in Z$. Therefore, an element $f \in \hom_{\cO_Z}(\Omega_{Z},L)$ is the same as an element $\{f_i\} \in \bigoplus_{i \in I}\hom_{\cO_{Z_i}}(\Omega_{Z_i},L_i)$ such that $f_i(n)=0$ for every $n \in N(Z) \cap Z_i$.
\end{proof}

\subsection*{Universally closedness}

This section is dedicated to prove that $\eta$ is universally closed. The valuative criterion tells us that $\eta$ is universally closed if and only if for every diagram 
$$
\begin{tikzcd}
\spec Q \arrow[r] \arrow[d, hook] & \Htilde_g^r \arrow[d, "\eta"] \\
\spec R \arrow[r]                 & \Mtilde_g^r                  
\end{tikzcd}
$$
where $R$ is a complete DVR with algebraically closed residue field $k$ and field of fractions $K$, there exists a lifting 
$$
\begin{tikzcd}
\spec K \arrow[r] \arrow[d, hook]    & \Htilde_g^r \arrow[d, "\eta"] \\
\spec R \arrow[r] \arrow[ru, dashed] & \Mtilde_g^r                  
\end{tikzcd}
$$
which makes everything commutes. 

This amounts to extending the hyperelliptic involution from the general fiber of a family of $A_r$-stable curves over a DVR to the whole family. The precise statement is the following. 

\begin{theorem}\label{theo:univ-closed}
	Let $C_R\rightarrow \spec R$ be a family of $A_r$-stable curves over $R$, which is a complete DVR with fraction field $K$ and algebraically closed residue field $k$. Suppose there exists a hyperelliptic involution $\sigma_K$ of the general fiber $C_K \rightarrow \spec K$. Then there exists a unique hyperelliptic involution $\sigma_R$ of the whole family $C_R\rightarrow \spec R$ such that $\sigma_R$ restrict to $\sigma_K$ over the general fiber.
\end{theorem} 

\begin{remark}
	In the ($A_1$)-stable case, \Cref{theo:univ-closed} follows from the finiteness of the inertia, as $\overline{\mathcal{M}}_g$ is a Deligne-Mumford separated (in fact proper) stacks. However if $r\geq 2$, the inertia of $\Mtilde_g^r$ is not proper. As a matter of fact, $I_{\Mtilde_g^r}[2]$, i.e. the $2$-torsion of $I_{\Mtilde_g^r}$, is not finite over $\Mtilde_g^r$ either. It is clearly quasi-finite, but the properness fails as the following example shows. 
	
	Let $k=\CC$ and $\AA^1$ be the affine line over the complex number. We consider the following situation: consider the projective line $p_1:\PP^1\times \AA^1\rightarrow \AA^1$ over $\AA^1$ and the line bundle $p_2^*\cO_{\PP^1}(-4)$ (where $p_2:\PP^1\times \AA^1\rightarrow \PP^1$ is the natural projection). Every non-zero section $f\in p_{1,*}p_2^*\cO_{\PP^1}(8)\simeq \H^0(\PP^1,\cO(8)) \otimes_{\CC} \cO_{\AA^1}$ gives rise to a cyclic cover of 
	$$\begin{tikzcd}
	C_f\arrow[rd] \arrow[rr, "2:1"] &       & \PP^1\times\AA^1 \arrow[ld, "p_1"] \\
	& \AA^1 &                                   
	\end{tikzcd}$$
	 such that $C_f\rightarrow \AA^1$ is a family of (arithmetic) genus $3$ curves (see \cite{ArVis}). One can prove that they are all $A_r$-stable if $r\geq 8$. Furthermore, every automorphism of the data $(\PP^1\times \AA^1,\cO(-4),f)$ gives us an automorphism of $C$ which commutes with the cyclic cover map. We have already proven that the association is fully faithful and these are the only possible automorphisms. 
	 
	 Consider the section $f:=x_0^2x_1^2(x_0-x_1)^2(tx_0-x_1)^2 \in \H^0(\PP^1,\cO(8))\otimes \CC[t]$ where $[x_0:x_1]$ are the homogeneous coordinates of the projective line. Then it is easy to show that the family $C_f$ lives in $\Mtilde_3^3$. If $t\neq 0,1$, we have constructed a $A_1$-stable genus $3$ curve which can be obtained by gluing two projective line in four points with the same cross ratio, which is exactly $t$. Whereas if $t$ is either $0$ or $1$, we obtain a genus $3$ curve which is $A_3$-stable, as two of the four nodes in the generic case collapse into a tacnode. 
	 
	 Let $\phi_t$ be an element of $\PGL_2(\CC(t))$ defined by the matrix
	 $$
	 \begin{bmatrix}
	 0  &  1 \\
	 t  &  0
	 \end{bmatrix},
	 $$
	 hence an easy computation shows that $\phi_t$ is an involution of the data $$(\PP^1_{\CC(t)},\cO(-4),f)$$ and therefore of $C_f\otimes_{\AA^1}\spec \CC(t)$. First of all, this is not hyperelliptic: the quotient of $C_f\otimes_{\AA^1}\spec \CC(t)$ by this involution is a genus $1$ curve geometrically obtained by intersecting two projective lines in two points. Furthermore, it does not have a limit for $t=0$. This example shows us that we really need the hyperelliptic condition to lift an involution.
\end{remark}

First of all, we can reduce to considering morphisms $\spec K \rightarrow \Htilde_g^r$ which land in a dense open of $\Htilde_g^r$. Therefore \Cref{prop:smooth-hyp} implies that it is enough to prove the theorem when the family $C_R \rightarrow \spec R$ is generically smooth. In particular this implies $C_R$ is a $2$-dimensional normal scheme.

Using the normality of $C_R$ and the properness of the morphism $C_R \rightarrow \spec R$, one can prove that $\sigma_K$ can be uniquelly extended to an open $U$ of $C_R$ whose complement has $\codim 2$. Let us call $\sigma':C_R \dashrightarrow C_R$ the extension of $\sigma_K$ to the open $U$. 

\begin{lemma}\label{lem:contract}
	In the situation above, suppose that $\sigma'$ does not contract any one-dimensional subscheme of $C_R$ to a point. Then there exists an extension of $\sigma'$ to a regular morphism $\sigma_R:C_R\rightarrow C_R$ which is a hyperelliptic involution. 
\end{lemma}

\begin{proof}
	Let $U\subset C_R$ be the open subscheme where $\sigma'$ is defined and let $\Theta\subset C_R\times_{\spec R} C_R$ be the closure of the graph $U\hookrightarrow C_R\times_{\spec R} C_R$ associated to $\sigma'$. We denote by $p_i$ the restriction of the $i$-th projection $C_R\times_{\spec R} C_R \rightarrow C_R$ to $\Theta$ for $i=1,2$. Suppose there exists a one-dimensional irreducible scheme $\Gamma\subset \Theta$ such that $p_2(\Gamma)$ is just a point, then $p_1(\Gamma)$ has to be one-dimensional but this would imply that $\sigma'$ contracts it. Therefore $p_2$ is quasi-finite and proper, thus finite. We have then a finite birational morphism to a normal variety, therefore it is an isomorphism. Hence we have that $\sigma'^{-1}$ can be defined over $C_R$, i.e. it is a regular morphism. We denote it by $\widetilde{\sigma}$. Because both $\widetilde{\sigma}$ and $\sigma'$ restricts to $\sigma_K$ at the generic fiber, we have that $\widetilde{\sigma}$ is in fact an extension of $\sigma'$ to a regular morphism and it is an involution. The hyperelliptic property follows from \Cref{prop:open-closed-imm}.
\end{proof}

The previous lemma implies that it is enough to prove that $\sigma'$ does not contract any one-dimensional subscheme of $C_R$. To do this, we use a classical but fundamental fact for smooth hyperelliptic curves. 

\begin{lemma}\label{lem:can-com}
	Let $C$ be a smooth hyperelliptic curve of genus $g$ over an algebraically closed field. Then the canonical morphism $$\phi_{|\omega_C|}:C\longrightarrow \PP^{g-1}$$
	factors through the hyperelliptic quotient.
\end{lemma}

Before continuing with the proof of \Cref{theo:univ-closed}, we recall why the genus $2$ case is special.

\begin{remark}	\label{rem:genus-2}
	If $g=2$, we need to prove that $\eta$ is an isomorphism.  In this situation, given a family $\pi:C\rightarrow S$ of $A_r$-stable curves, there is always a hyperelliptic involution. Indeed, one can consider the morphism $\phi_{|\omega_C^{\otimes 2}|}:C \rightarrow \PP(\pi_*\omega_C^{\otimes 2}) $. A straightforward computation proves that $\omega_C^{\otimes 2}$ is globally generated and the image $Z$ of the morphism associated to $\omega_C^{\otimes 2}$ is a family of curves of genus $0$. In fact we have a factorization of the morphism $\phi_{|\omega_C^{\otimes 2}|}$ through its image
	$$
	\begin{tikzcd}
	C \arrow[rd,"\pi"]  \arrow[r, "p"] & Z \arrow[d] \arrow[r, hook] & \PP(\pi_*\omega_C^{\otimes 2}) \arrow[ld]   \\
		& S              &                               
		\end{tikzcd}$$
	and one can construct an involution $\sigma$ of $C$ such that the quotient morphism is exactly $p$. 
	
	The same strategy cannot work for higher genus, as we know that $\eta$ is not surjective.  
\end{remark}

The idea is to use the canonical map to prove that the involution $\sigma'$ does not contract any one-dimensional subscheme of $C_R$, or equivalently any irreducible component of the special fiber $C_k:=C\otimes_R k$, where $k$ is the residue field of $R$. Indeed, $\sigma'$ commutes with the canonical map on a dense open, namely the generic fiber, because of \Cref{lem:can-com}. Reducedness and separatedness of $C_R$ imply that they commute wherever they are both defined. Let $\sigma'_k$ the restriction of $\sigma'$ to the special fiber and let $\Gamma$ be an irreducible component of $C_k$. Suppose we have the two following properties:
\begin{itemize}
	\item[1)] the open of definition of the canonical morphism of $C_k$ intersects $\Gamma$,
	\item[2)] the canonical morphism of $C_k$ does not contract $\Gamma$ to a point;
\end{itemize}
thus $\sigma'_k$ does not contract $\Gamma$ to a point, and neither does $\sigma'$. 

In the rest of the section, we describe the base point locus of $|\omega_C|$ and we prove that the canonical morphism contracts only a specific type of irreducible components. Then, we prove that $\sigma'$ does not contract these particular components using a variation of the canonical morphism. Thus we can apply \Cref{lem:contract} to get \Cref{theo:univ-closed}.

\begin{proposition}\label{prop:base-point-can}		
	Let $C/k$ be an $A_r$-stable genus $g$ curve over an algebraically closed field. The canonical map is defined on the complement of the set ${\rm SN}(C)$, which contains two type of closed points, namely:
	\begin{itemize}
		\item[$(1)$]   $p$ is a separating nodal point;
		\item[$(2)$]   $p$ belongs to an irreducible component of arithmetic genus $0$ which intersect the rest of curve in separating nodes.
	\end{itemize} 
\end{proposition} 
\begin{proof}
	Follows from Theorem D of \cite{Cat}.
\end{proof}

\Cref{prop:base-point-can} implies that there may be irreducible components of $C_k$ where the canonical map is not defined. We prove in \Cref{lem:not-poss-comp} that in fact they cannot occur in our situation.

Now we prove two lemmas which partially imply the previous one but we need them in this particular form. The first one is a characterization of the points of type (2) as in \Cref{prop:base-point-can}.

\begin{lemma}\label{lem:char-type2}
	Let $C$ be an $A_r$-prestable curve of genus $g\geq 1$ over $k$ and $p\in C$ be a smooth point. Then $\H^0(C,\cO(p))=2$ if and only if $p$ is of type $(2)$ as in \Cref{prop:base-point-can}.
\end{lemma}

\begin{proof}
	The \emph{if} part follows from an easy computation. Let us prove the \emph{only if} implication. Let $\Gamma$ be the irreducible component that contains $p$. Then $h^0(\Gamma,\cO(p))\geq h^0(C,\cO(p))$, which implies $h^0(\Gamma,\cO(p))=2$ or equivalently $\Gamma$ is a projective line. Because $\cO(p)\vert_{\tilde{\Gamma}}\simeq \cO_{\tilde{\Gamma}}$ for every irreducible component $\tilde{\Gamma}$ different from $\Gamma$, we have that $h^0(C,\cO(p))=2$ implies that for every connected subcurve $C'\subset C$ not cointaining $\Gamma$, we have that $C'\cap \Gamma$ is a $0$-dimensional subscheme of length at most $1$. This is equivalent to the fact that $\Gamma$ intersects $C-\Gamma$ only in separating nodes. 
\end{proof}

The second lemma helps us describe the canonical map of $C$ when restricted to its $A_1$-separating decomposition.

\begin{lemma}\label{lem:sep-decom-hodge}
Let $C/k$ be an $A_r$-stable curve of genus $g$ and let $\{C_i\}_{i \in I}$ its $A_1$-separating decomposition (see \Cref{def:sep-dec}). Then we have that 
$$\H^0(C,\omega_C)=\bigoplus_{i \in I}\H^0(C_i,\omega_{C_i}).$$	
\end{lemma}

\begin{proof}
	Let $n \in C$ be a separating nodal point and let $C_1$ and $C_2$ the two subcurves of $C$ such that $C=C_1\cup C_2$ and $C_1\cap C_2=\spec k(n)$. Then we have the exact sequence of coherent sheaves
	$$ 0 \rightarrow \omega_C \rightarrow \iota_{1,*}\omega_{C_1}(n)\oplus \iota_{2,*}\omega_{C_2}(n) \rightarrow k(n) \rightarrow 0$$
	where $\iota_i:C_i \hookrightarrow C$ is the natural immersion for $i=1,2$. Taking the global sections, $h^1(\omega_C)=1$ implies that $\H^0(C,\omega_C)=\H^0(C_1,\omega_{C_1}(n))\oplus\H^0(C_2,\omega_{C_2}(n))$. Finally, we claim that $\H^0(C,\omega_C(p))=\H^0(C,\omega_C)$ for every smooth point $p$ on a connected reduced Gorenstein curve $C$. The global sections of the following exact sequence 
	$$0 \rightarrow \cO_C(-p) \rightarrow \cO_C \rightarrow k(p) \rightarrow 0$$
	implies that the claim is equivalent to the vanishing of $\H^0(C,\cO_C(-p))$. This is an easy exercise.
\end{proof}
Let $\{C_i\}_{i \in I}$ be the $A_1$-separating decomposition of $C$ and $g_i$ be the genus of $C_i$. We will see in \Cref{lem:not-poss-comp} that there are no points of type $(2)$ as in \Cref{prop:base-point-can}. Therefore from now on, we suppose that there are no such points and in particular, $g_i>0$ for every $i \in I$.
The previous result implies that the composite 
$$
\begin{tikzcd}
C_i \arrow[r, "\iota_i", hook] & C \arrow[r, "\phi_{|\omega_C|}", dashed] & \PP^{g-1}
\end{tikzcd}
$$
factors through a map $f:C_i\dashrightarrow \PP^{g_i-1}$ induced by the vector space $\H^0(C_i,\omega_{C_i})$ inside the complete linear system of $\omega_{C_i}(\Sigma_i)$, where $\Sigma_i$ is the Cartier divisor on $C_i$ defined by the intersection of $C_i$ with the curve $C-C_i$, or equivalently the restriction of the separating nodal points on $C_i$. This follows because the restriction
$$ \H^0(C,\omega_C)  \longrightarrow \H^0(C_i,\omega_{C_i})$$
is surjective as long as there are no points of type $(2)$ as in \Cref{prop:base-point-can}  on $C_i$. Notice that the rational map $f$ is defined on the open $C_i\setminus \Sigma_i$, and we have that $f\vert_{C_i\setminus \Sigma_i}\equiv \phi_{|\omega_{C_i}|}\vert_{C_i\setminus \Sigma_i}$.

We have two cases to analyze: $C_i$ is either an $A_r$-stable genus $g_i$ curve or only $A_r$-prestable. In the first case, $C_i$ is an $A_r$-stable curve without separating nodes, and therefore the canonical morphism $\phi_{|\omega_{C_i}|}$ is globally defined and finite, because $\omega_{C_i}$ is ample. Therefore the canonical morphism of $C$ does not contract any irreducible component of the subcurve $C_i$. 

Suppose now that $C_i$ is not $A_r$-stable but only prestable and let $\Gamma$ be an irreducible component where $\omega_{C_i}\vert_{\Gamma}$ is not ample, in particular $g_{\Gamma}\leq 1$. In this case, the canonical morphism is not helpful.

\begin{remark}
	Let $C$ be a ($A_1$)-stable genus $2$ curve such that it is constructed as intersecting two smooth genus $1$ curves in a separating node. One can prove easily that the canonical morphism is trivial. Similarly, if we have a genus $g$ stable curve constructed by intersecting smooth genus $1$ curves in separating nodes, the same is still true.
\end{remark}
 
The idea is to construct a variation of the canonical morphism which does not contract the component $\Gamma$ and it still commutes with the involution. The following lemma gives us a possible candidate.

\begin{lemma}\label{lem:var-can}
	Let $C_R\rightarrow \spec R$ be a generically smooth family of $A_r$-stable genus $g$ curves over $R$, a discrete valuation ring, $p \in C_k$ be a separating node of the special fiber of the family and $C_1$ and $C_2$ be the two subcurves of $C_k$ such that $C_k=C_1\cup C_2$ and $C_1\cap C_2=\{p\}$.  Then there exists an integer $m$ such that the followings are true:
	\begin{itemize}
		\item[(i)] $mC_1$ and $mC_2$ are Cartier divisors of $C_R$,
		\item[(ii)] (if we denote by $\cO(mC_1)$ and $\cO(mC_2)$ the induced line bundles) we have $\cO(mC_2)\vert_{C_1}=\cO_{C_1}(p)$  and $\cO(mC_2) \vert_{C_2}=\cO_{C_2}(-p)$, \item[(iii)] the line bundle $\omega_{C_R}(mC_i)$ verifies base change on $\spec R$.
	\end{itemize}
\end{lemma} 

\begin{proof}
	The existence of an integer $m$ that verifies (i) follows from the theory of Du Val singularities, as a separating node in a normal surface \'etale locally is of the form $y^2+x^2+t^m$ where $t$ is a uniformizer of the DVR. An \'etale-local computation shows that $\cO(mC_1)\vert_{C_2}=\cO_{C_2}(p)$ and because as line bundles $\cO(-mC_1-mC_2)$ is the pullback from the DVR of the ideal $(t^m)$ we get $\cO(mC_2)\vert_{C_2}=\cO_{C_2}(-p)$. 
	
	Because $R$ is reduced, to prove $(iii)$ it is enough to prove that $$\dim_k\H^0(C_k,\omega_{C/R}(mC_i)\vert_{C_k})=g$$
	and then use Grauert's theorem. 
	Without loss of generality, we can suppose $i=1$.
	Let us denote by $\cL$ the line bundle $\omega_{C/R}(mC_1)$ restricted to $C_k$. Then we have the usual exact sequence
	$$0\rightarrow \H^0(C_k,\cL)\rightarrow \H^0(C_1,\cL\vert_{C_1})\oplus \H^0(C_2,\cL\vert_{C_2}) \rightarrow \cL(p)$$ 
	where we have that $\cL\vert_{C_1}\simeq \omega_{C_1}$ while $\cL\vert_{C_2}\simeq \omega_{C_2}(2p)$. Because $\omega_{C_2}(2p)$ is globally generated in $p$, we get that $$h^0(\cL)=h^0(\omega_{C_1})+h^0(\omega_{C_2}(2p))-1=h^0(\omega_{C_1})+h^0(\omega_{C_2})=g_1+g_2$$ where $g_i$ is the arithmetic genus of $C_i$ for $i=1,2$. Because $p$ is a separating node, thus $g_1+g_2=g$ and we are done.
	
\end{proof}

\begin{remark}
In the situation of \Cref{lem:var-can}, the involution $\sigma'$ commutes with the rational map associated with the complete linear system of $\omega_{C_R}(mC_i)$ for $i=1,2$. This follows from the fact that $\omega_{C_R}(mC_i)$ restrict to the canonical line bundle in the generic fiber. 
\end{remark}

  If $g_{\Gamma}=1$, then $C_i=\Gamma$ (because it is connected). If $g_{\Gamma}=0$, we have that the intersection $\Gamma \cap (C_i-\Gamma)$ has length $2$ with $C_i-\Gamma$ connected subcurve of $C_i$ (because $g_{C_i}>0$). Let $p$ be a separating node in $\Gamma \cap (C-\Gamma)$. We apply \Cref{lem:var-can} and let us denote by $C_1$ the subcurve containing $\Gamma$. For the sake of notation, we define $\cL_{\Gamma}:=\omega_{C_R}(mC_2)$. 

\begin{proposition}
	In the situation above, the morphism $\phi_{\cL_{\Gamma}}$ induced by the complete linear system of $\cL_{\Gamma}$ does not contract $\Gamma$. 
\end{proposition}

\begin{proof} 
	We need to prove that the morphism $\phi_{\cL_{\Gamma}}$ restricted to the special fiber does not contract $\Gamma$. Using the description of $\H^0(C_k,\cL_{\Gamma}\vert_{C_k})$ of \Cref{lem:var-can}, we can prove an explicit description of $\phi_{\cL_{\Gamma}}$ in both cases. 
	
	If $g_{\Gamma}=1$, then the restriction of $\phi_{\cL_{\Gamma}}$ to $\Gamma$ is the map induced by the vector space $\H^0(\Gamma,\omega_{\Gamma}(2p))$ inside the complete linear system of $\omega_{\Gamma}(2p + (\Sigma_{\Gamma}-p))$ where $\Sigma_{\Gamma}$ is the Cartier divisor on $\Gamma$ associated with the intersection $\Gamma \cap (C-\Gamma)$. We are using that there are no points of type $(2)$ as in \Cref{prop:base-point-can}, which implies that the morphism 
	$$ \H^0(C,\phi_{\cL}) \rightarrow \H^0(\Gamma,\phi_{\cL_{\Gamma}})$$
	is surjective. Similarly as we have done before, we have that $\phi_{\cL_{\Gamma}}\vert_{\Gamma}$ is defined everywhere except in support of the Cartier divisor $\Sigma_{\Gamma}-p$ and it coincides with $\omega_{\Gamma}(2p)$ where it is defined . This implies $\Gamma$ is not contracted to a point. 
	
	Finally, if $g_{\Gamma}=0$, let $D$ be the Cartier divisor on $\Gamma$ associated with the intersection $\Gamma\cap(C_i-\Gamma)$ of length $2$ and by ${\rm res}_D$ one of the following morphisms: 
	\begin{itemize}
		\item ${\rm res}_{q_1}+{\rm res}_{q_2}$ when $D=q_1+q_2$, 
		\item ${\rm res}_q$ when $D=2q$.
	\end{itemize}
	We have that $\phi_{\cL_{\Gamma}}$ restricted to $\Gamma$ is the morphism induced by kernel $V$ of 
	$$ {\rm res}_{D}(2p+(\Sigma-p)):\H^0(\Gamma,\omega_{\Gamma}(2p+(\Sigma_{\Gamma}-p)+D)) \longrightarrow k$$
    where the morphism of vector spaces ${\rm res}_{D}(2p+(\Sigma-p))$ is the tensor of the morphism ${\rm res}_D$ by the line bundle $\cO(2p+(\Sigma-p))$. This again is a conseguence of \Cref{lem:not-poss-comp}. As before, we can prove that $\phi_V$ is defined everywhere except in the support of $\Sigma-p$ and it coincides with the morphism associated with the kernel $W$ of the morphism 
    $$ {\rm res}_{D}(2p):\H^0(\Gamma,\omega_{\Gamma}(2p+D)) \longrightarrow k$$
	where the morphism of vector spaces ${\rm res}_{D}(2p)$ is the tensor of the morphism ${\rm res}_D$ by the line bundle $\cO(2p)$. A straightforward computation shows that $\phi_W$ does not contract $\Gamma$.
\end{proof}

\begin{corollary}
 In the situation above, $\sigma'$ does not contract $\Gamma$.
\end{corollary}

\begin{proof}
 This follows because $\phi_{\cL_{\Gamma}} \circ \sigma' = \phi_{\cL_{\Gamma}}$ and $\phi_{\cL_{\Gamma}}$ does not contract $\Gamma$.
\end{proof}

	We have assumed that there are no points of type $(2)$ as in \Cref{prop:base-point-can}. The last lemma explain why this assumption can be made. We decided to leave it as the last lemma because it is the only statement where we use a result for the $A_1$-stable case which we cannot prove indipendently for the $A_r$-stable case. ´
	
	\begin{lemma}\label{lem:not-poss-comp}
		Let $C\rightarrow \spec R$ be an $A_r$-stable curve of genus $g$ and suppose the generic fiber is smooth and hyperelliptic. Then if $C_k$ is the special fiber, there are no point $p\in C_k$ of type $(2)$ as in \Cref{prop:base-point-can}. 
	\end{lemma}
\begin{proof} 
	The proof relies on the fact that the stable reduction for the $A_n$-singularities is well-know, see \cite{Hass} and \cite{CasMarLaz}. Let $C_K$ be the generic fiber, then we know that there exists, a family $\widetilde{C}$ such that the special fiber $\widetilde{C}_k$ is stable. The explicit description of the stable reduction implies that if there is a component $\Gamma$ in $C_k$ of genus $0$ intersecting the rest of curve only in separating nodes, then the same component should appear in the stable reduction. But this is absurd, as the stable limit of a smooth hyperelliptic curve is a stable hyperelliptic curve, which does not have such irreducible component.
\end{proof}

\bibliographystyle{plain}
\bibliography{Bibliografia}

\end{document}